\documentclass[12pt,a4paper,leqno]{amsart}
\usepackage[margin=2.9cm]{geometry}
\usepackage{amsmath,amscd,amssymb,amsthm}
\usepackage{mathrsfs}
\usepackage{setspace,graphicx}
\usepackage{multirow}
\usepackage{tikz}
\usepackage{enumerate}
\usepackage[utf8]{inputenc}
\usepackage[english]{babel}
\usepackage[square,numbers]{natbib}

\theoremstyle{plain}
\newtheorem{Th}{Theorem}[section]
\newtheorem{Lemma}[Th]{Lemma}
\newtheorem{Cor}[Th]{Corollary}

\theoremstyle{definition}
\newtheorem{Def}[Th]{Definition}

\newtheorem{Rem}[Th]{Remark}
\newtheorem{?}[Th]{Problem}
\newtheorem{Ex}[Th]{Example}

\newcommand{\Addresses}{{
  \bigskip
  \footnotesize

  \noindent Zheng Xiao, \textsc{Beijing International Center for Mathematical Research, Peking University,
    Beijing, China 100871}\\
    \noindent \textit{E-mail address}: \texttt{xiaozheng@bicmr.pku.edu.cn}

}}

\title[GREATEST COMMON DIVISORS IN ALGEBRAIC NUMBERS AND LINEAR RECURRENCE]{GREATEST COMMON DIVISORS FOR POLYNOMIALS IN ALMOST UNITS AND APPLICATIONS TO LINEAR RECURRENCE SEQUENCES}
\author{ZHENG XIAO}

\begin{document}

\maketitle


\begin{abstract} We bound the greatest common divisor of two coprime multivariable polynomials evaluated at algebraic numbers, generalizing work of Levin, and going towards conjectured inequalities of Silverman and Vojta. As an application, we prove results on greatest common divisors of terms from two linear recurrence sequences, extending the results of Levin, who considered the case where the linear recurrences are simple, and improving recent results of Grieve and Wang. The proofs rely on Schmidt's Subspace Theorem.
\end{abstract}


\section{Introduction}
\subsection{Main results}
In recent work of Levin \cite{Le}, the following result was proven, giving an inequality for greatest common divisors of polynomials evaluated at $S$-unit points.
\begin{Th}[Levin \cite{Le}]\label{gcdle}
Let $n$ be a positive integer. Let $\Gamma \subset \mathbb{G}_{m}^{n}(\bar{\mathbb{Q}})$ be a finitely generated group. Let $f(x_{1},\ldots,x_{n})$, $g(x_{1},\ldots,x_{n}) \in \bar{\mathbb{Q}}[x_{1},\ldots,x_{n}]$ be non-constant coprime polynomials such that not both of them vanish at $(0,\ldots,0)$. Let $h(\alpha)$ denote the (absolute logarithmic) height of an algebraic number $\alpha$. For all $\epsilon >0$, there exists a finite union $Z$ of translates of proper algebraic subgroups of $\mathbb{G}_{m}^{n}$ such that
$$\log \gcd (f(u_{1},\ldots,u_{n}),g(u_{1},\ldots,u_{n}))<\epsilon \max\{h(u_{1}),\ldots,h(u_{n})\}$$
for all $(u_{1},\ldots,u_{n}) \in \Gamma \setminus Z$.\\
\end{Th}

In particular, $\Gamma$ in Theorem \ref{gcdle} can be taken as the full set of $n$-tuples of $S$-units in a number field $k$, where $S$ is a finite set places of $k$ containing the archimedean places. In the above statement, $\log \gcd$ is the generalized logarithmic greatest common divisor, which is defined in Section \ref{pregcd}.\\

We want to generalize Theorem \ref{gcdle} beyond the setting of $S$-units points. To achieve this goal, we introduce the definition of almost $S$-units: Roughly speaking, an almost $(S,\delta)$-unit for some set of places $S$ in a number field $k$ is an element $u \in k$ whose dominant part of its height is due to an $S$-unit. 

\begin{Def} \label{asu}
For a fixed $\delta > 0$ and a fixed set of places $S$, if $u \in k^{*}$, then we say $u$ is an almost $(S,\delta)$-unit if 
$$h_{\bar S}(u):=\displaystyle \sum_{v \notin S}\lambda_{v}(u)+\lambda_{v}\left(\dfrac{1}{u}\right) \leq \delta h(u)$$
(see Section 2.1 for the definition of $\lambda_{v}$). We denote the set of all almost $(S,\delta)$-units by $k_{S,\delta}$. More generally, let 
$$\mathbb{G}_m^n(k)_{S,\delta}:=\{\textbf{u} \in \mathbb{G}_{m}^{n}(k)|h_{\bar{S}}(\textbf{u})\leq \delta h(\textbf{u})\},$$ 
where $$h_{\bar S}(\textbf{u})=\displaystyle \sum_{v \notin S}\lambda_{v}(\textbf{u})+\lambda_{v}\left(\dfrac{1}{\textbf{u}}\right).$$ 
\end{Def}

With Definition \ref{asu}, we prove the following generalization of Theorem \ref{gcdle}, which shows that $\Gamma=(\mathcal{O}_{k,S}^{*})^{n}$ may be ``thickened" to $\mathbb{G}_{m}^{n}(k)_{S,\delta}$ for some positive $\delta$ (depending on $\epsilon$). 
\begin{Th}\label{combined}(Corollary \ref{ins})
Let $n$ be a positive integer and $k$ a number field, $f(x_{1},\ldots,x_{n})$, $g(x_{1},\ldots, x_{n}) \in k[x_{1},\ldots,x_{n}]$ be nonconstant coprime polynomials such that not both of them vanish at $(0,\ldots,0)$. For all $\epsilon >0$, there exists $\delta >0$ and a proper Zariski closed subset $Z$ of $\mathbb{G}_{m}^{n}$ such that:\\
\centerline{$\log \gcd(f(u_{1},\ldots,u_{n}),g(u_{1},\ldots,u_{n})) < \epsilon \max\{h(u_{1}),\ldots,h(u_{n})\}$}\\
for all $(u_{1},\ldots,u_{n}) \in \mathbb{G}_m^n(k)_{S,\delta} \setminus Z$.
\end{Th}

By Theorem 5 of \cite{Ev}, we may further choose $Z$ so that it is a (possibly infinite) union of positive-dimensional torus cosets.

In fact, we prove the following refinement of Theorem \ref{combined}. Theorem \ref{gcdcomb} extends Levin's Theorem \ref{gcdle} from integral points to rational points, and may be viewed as progress towards Vojta's conjecture for certain blown-up varieties. More precisely, in Remark \ref{Si} we discuss the relation of Theorem \ref{gcdcomb} with conjectured inequalities of Silverman based on Vojta's conjecture.\\

\begin{Th}\label{gcdgen}(Theorem \ref{gcdcomb})
Let $k$ be a number field and let $S$ be a finite set of places of $k$ containing the archimedean places. Let $f$, $g \in k[x_{1}, \ldots, x_{n}]$ be coprime polynomials that don't both vanish at the origin $(0,\ldots,0)$. For all $0< \delta <1$, there exists a proper Zariski closed subset $Z$ of $\mathbb{G}_{m}^{n}$ such that
$$\displaystyle \log \gcd(f(u_{1},\ldots, u_{n}), g(u_{1},\ldots,u_{n}))< C \delta^{1/2}\sum_{i=1}^{n}h(u_{i})$$
for all $\textbf{u}=(u_{1},\ldots,u_{n})\in \mathbb{G}_{m}^{n}(k)_{S,\delta} \setminus Z$ satisfying $h_{\bar{S}}(\textbf{u}) < \delta h(\textbf{u})$, where $C=6(\deg f+\deg g)n^{2}$ is a constant.\\
\end{Th}

This Theorem gives a G.C.D. inequality of the form similar to what Vojta's Conjecture predicts. Assuming Vojta's Conjecture, Silverman  obtained an upper bound for the polynomial G.C.D. in \cite{Si}. By properly extending the notions from $\mathbb{Q}$ to a number field, we can compare Silverman's conjectural upper bound with our inequality. This is discussed in detail in Remark \ref{Si}. We also note work of Grieve \cite{Gr} in this direction.\\

As an application, we state our main result on linear recurrence sequences: \\

\begin{Th}\label{recfin}
Let\\
$$\displaystyle F(m)=\sum_{i=1}^{s}p_{i}(m)\alpha_{i}^{m},$$
$$\displaystyle G(n)=\sum_{j=1}^{t}q_{j}(n)\beta_{j}^{n},$$
define two algebraic linear recurrence sequences. Let $k$ be a number field such that all coefficients of $p_{i}$ and $q_{j}$ and $\alpha_{i}$, $\beta_{j}$ are in $k$, for $i=1,\ldots,s$, $j=1,\ldots,t$. Let
$$S_{0}=\{v \in M_{k}:\max\{|\alpha_{1}|_{v},\ldots,|\alpha_{s}|_{v},|\beta_{1}|_{v},\ldots,|\beta_{t}|_{v}\} < 1\}.$$
Then all but finitely many solutions $(m,n)$ of the inequality:\\
$$\displaystyle \sum_{v \in M_{k} \setminus S_{0}}-\log^{-}\max\{|F(m)|_{v},|G(n)|_{v}\} > \epsilon \max\{m,n\}$$
are of the form:\\
\centerline{$(m,n)=(a_{i}t,b_{i}t)+(\mu_{1},\mu_{2})$,\ \ \ \  $\mu_{1},\mu_{2} \ll \log t,\ t \in \mathbb{N},\  i=1, \ldots, r$}\\
with finitely many choices of nonzero integers $(a_{i},b_{i})$ .\\
Moreover, if the roots of $F$ and $G$ are independent (see Def \ref{indroot}), then the solutions $(m,n)$ satisfy one of the finitely many linear relations:\\
\centerline{$(m,n) = (a_{i}t+b_{i}, c_{i}t+d_{i})$, $t \in \mathbb{N}, i=1, \ldots, r$}\\
where $a_{i},b_{i},c_{i},d_{i} \in \mathbb{N}, a_{i}c_{i} \neq 0$, and the linear recurrences $F(a_{i}\bullet+b_{i})$ and $G(c_{i}\bullet+d_{i})$ have a nontrivial common factor for $i=1,\ldots,r$.
\end{Th}

\begin{Ex}
Under the set up of Theorem \ref{recfin}, we give an example illustrating the necessity of $(\mu_{1},\mu_{2})$ in the statement:\\

Define the two linear recurrence sequences as: $F(m)=mp^{m}+1$, $G(n)=p^{n}+1$, where $p$ is a prime. In the notations of Theorem \ref{recfin}, $S_{0}=\emptyset$. Let $\epsilon <\log 2$. It is easily seen that for $(m,n)=(p^{k},p^{k}+k)$, $\forall k \in \mathbb{Z}_{>0}$, $F(m)=p^{p^{k}+k}+1=G(n)$, so the inequality
$$\displaystyle \log\gcd\{|F(m)|,|G(n)|\}=\log (p^{p^{k}+k}+1)>\epsilon (p^{k}+k)=\epsilon \max\{m,n\}$$
holds for infinitely many $k$ and hence infinitely many $(m,n)$. 
It is easily seen that such pairs $(m,n)$ do not lie on finitely many lines, but do lie in a logarithmic region around the line $x=y$, i.e., for such pairs we may write $(m,n)=(t,t)+(\mu_{1},\mu_{2})$ with $\mu_{1},\mu_{2}\ll \log t$ in agreement with Theorem \ref{recfin}.\\
\end{Ex}

The above example not only indicates the necessity of $(\mu_1,\mu_2)$, but also inspires us to think further: if every exceptional case is of the form similar to the example? Here we give a positive answer:
\begin{Th}[Theorem \ref{finimp}]
For all but finitely many solutions to the inequality in Theorem \ref{recfin}, there are finitely many choices of nonzero integers $(a_{i},b_{i}, c_{i}, d_{i}),\  a_{i}c_{i} \neq 0$ such that all but finitely many solutions $(m,n) \in \mathbb{N}^{2}$ of the inequality
$$\displaystyle \sum_{v \in M_{k} \setminus S_{0}}-\log^{-}\max\{|F(m)|_{v},|G(n)|_{v}\} > \epsilon \max\{m,n\},$$
either satisfies finitely many linear relations:
$$(m,n)=(a_{i}t+b_{i},c_{i}t+d_{i}), i=1,\ldots,r,$$
or there exist a pair of constants $(a,b)$ with $T:=|am+bn| \ll O(\max\{\log m,\log n\})$ and linear recurrences $f$ and $g$ indexed by $T$ such that $m=f(T)$ and $n=g(T)$. \\
Moreover, if $\{u_i\}_{i=1,\ldots,y_r}$ is a set of the combined roots of $F,G,m,n$ such that $F$ and $G$ can be written as polynomials in variables $T,x_1,\ldots,x_r,y_1,\ldots, y_r$, where $x_i=u_i^T$ and $y_i=u_i^m$. Then $F,G$ admits a non-trivial common divisor in $k[T,x_1,\ldots,x_r,y_1,\ldots,y_r]$.
\end{Th}

\subsection{Background}
Upper bounds for the greatest common divisor of integers of the form $a^n-1$ and $b^n-1$ were first studied by Bugeaud, Corvaja, and Zannier in \cite{BCZ}, where they proved the following inequality:

\begin{Th}[Bugeaud, Corvaja, Zannier \cite{BCZ}]\label{gcd1}
Let $a$, $b$ be multiplicatively independent integers, and let $\epsilon > 0$. Then, provided $n$ sufficiently large, we have
$$\log \gcd(a^{n}-1, b^{n}-1) < \epsilon n.$$
\end{Th} 
Note that even though the statement is simple, the proof requires Schmidt's Subspace Theorem from Diophantine approximation. Actually, for most of the following works, it is the fundamental ingredient in their proofs.\\

Corvaja, Zannier \cite{CZ} and Hern{\'a}ndez, Luca  \cite{HL} subsequently extended Theorem \ref{gcd1} to $S$-unit integers:

\begin{Th}[Corvaja, Zannier \cite{CZ} and Hern{\'a}ndez, Luca \cite{HL}]\label{gcds}
Let $p_{1},\ldots,p_{t} \in \mathbb{Z}$ be prime numbers and let $S=\{\infty, p_{1},\ldots,p_{t}\}$. Then for every $\epsilon >0$,
$$\log \gcd (u-1,v-1)\leq \epsilon \max\{\log|u|,\log|v|\}$$
for all but finitely many multiplicatively independent $S$-unit integers $u,v \in \mathbb{Z}_{S}^{*}$. 
\end{Th}

More generally,  Corvaja and Zannier proved an inequality in the case of bivariate polynomials.
\begin{Th}[Corvaja, Zannier \cite{CZ1}]
Let $\Gamma \subset \mathbb{G}_{m}^{2}(\bar{\mathbb{Q}})$ be a finitely generated group. Let $f(x,y), g(x,y) \in \bar{\mathbb{Q}}[x,y]$ be nonconstant coprime polynomials such that not both of them vanish at $(0,0)$. For all $\epsilon >0$, there exists a finite union $Z$ of translates of proper algebraic subgroups of $\mathbb{G}_{m}^{2}$ such that
$$\log \gcd (f(u,v),g(u,v)) < \epsilon \max\{h(u),h(v)\}$$
for all $(u,v) \in \Gamma \setminus Z$.
\end{Th}

Levin \cite{Le} then generalized Corvaja-Zannier's theorem to higher-dimensions, which is stated in Theorem \ref{gcdle}.

In a slightly different direction, Luca \cite{Lu} extended Theorem \ref{gcds} to rational numbers $u$ and $v$ that are ``close" to being an $S$-unit. Let $u$ be a non-zero rational number, and $S$ a fixed finite set of primes. We may write $u$ uniquely, up to a sign, in the form $u=u_{S}\cdot u_{\bar{S}}$, where $u_{S}$ is a rational number in reduced form having both its numerator and denominator composed of primes in $S$, and $u_{\bar{S}}$ is a rational number in reduced form having both its numerator and denominator free of primes from $S$. Luca proved the following:\\

\begin{Th}[Luca \cite{Lu}]\label{lura}
Let $S$ be a finite set of places of $\mathbb{Q}$. For $\epsilon >0$, there exist three positive constants $K_{1}$, $K_{2}$, $K_{3}$ depending on $S$ and $\epsilon$, such that for any rational numbers $u$ and $v$ satisfying
$$\log \gcd(u-1,v-1) \geq \epsilon \max \{h(u),h(v)\},$$
one of the following three conditions holds:
\begin{enumerate}[(i)]
\item$\max \{h_{rat}(u),h_{rat}(v)\} < K_{1}$,
\item$u^{i}=v^{j}$ with $\max \{|i|,|j|\} <K_{2}$,
\item$\max \{h_{\bar{S}}(u),h_{\bar{S}}(v)\} > K_{3}\dfrac{h}{\log h}$,
\end{enumerate}
where $h_{\bar{S}}(u)=h(u_{\bar{S}})$, $h_{rat}\left(\dfrac{x}{y}\right)=\max\left\{\dfrac{h(x)}{h(y)},\dfrac{h(y)}{h(x)}\right\}$ and $h=\max\{h(u),h(v)\}$. 
\end{Th}
This shows the G.C.D. of two rational integers $u-1$ and $v-1$ cannot be large unless $u$ and $v$ are multiplicatively dependent or have large non-$S$ height. One main theorem of this paper (Corollary \ref{notins}) can also be viewed as a generalization of Theorem \ref{gcdle} along the lines of Luca's theorem.\\



On the other hand, Levin \cite{Le} also gave a classification (Theorem \ref{lrsle}) of large G.C.D.s among terms from simple linear recurrence sequences (see also earlier work of Fuchs \cite{F}). A primary goal of our work is to study the case of general linear recurrences (i.e., without the assumption that the linear recurrence is simple). In the case of binary linear recurrences, Luca \cite{Lu} showed:

\begin{Th}[Luca \cite{Lu}]\label{lubin}
Let $a$ and $b$ be non-zero integers which are multiplicatively independent, and let $f$, $g$, $f_{1}$ and $g_{1}$ be non-zero polynomials with integer coefficients. For every positive integer $n$ set
$$u_{n}=f(n)a^{n}+g(n)$$
and
$$v_{n}=f_{1}(n)b^{n}+g_{1}(n).$$
Then, for every fixed $\epsilon> 0$ there exists a positive constant $C_{\epsilon} > 0$ depending on $\epsilon$ and on the given data $a, b, f , f_{1}, g$ and $g_{1}$, such that
$$\log \gcd (u_{n},v_{m}) <\epsilon \max \{m,n\}$$
holds for all pairs of positive integers $(m,n)$ with $\max \{m, n\} > C_{\epsilon}$.
\end{Th}
The essential assumption is that $a$ and $b$ are multiplicatively independent integers, which gives a contradiction to condition (ii) of Theorem \ref{lura}. Note that Theorem \ref{lubin} is proved without the assistance of Schmidt's Subspace Theorem, so one should expect that a stronger result can be proved with the Subspace Theorem applied to general linear recurrence sequences. In fact, a recent result due to Grieve and Wang \cite{GW} on general linear recurrences generalized Luca's binary case, and recovered Levin's result \ref{lrsle} at the same time. We will give an alternative proof of this theorem later.\\

Levin \cite{Le} applied Theorem \ref{gcdle} to terms from simple linear recurrence sequences, giving a classification of when two such terms may have a large G.C.D..
\begin{Th}[Levin \cite{Le}]\label{lrsle}
Let
$$\displaystyle F(m)=\sum_{i=1}^{s}c_{i}\alpha_{i}^{m},$$
$$\displaystyle G(n)=\sum_{j=1}^{t}d_{j}\beta_{j}^{n},$$
define two algebraic simple linear recurrence sequences. Let $k$ be a number field such that $c_{i},\alpha_{i},d_{j},\beta_{j} \in k$ for $i=1,\ldots,s$, $j=1,\ldots,t$. Let $M_{k}$ be the canonical set of places in $k$. Let
$$S_{0}=\{v \in M_{k}:\max\{|\alpha_{1}|_{v},\ldots,|\alpha_{s}|_{v},|\beta_{1}|_{v},\ldots,|\beta_{t}|_{v}\}<1\}.$$
Let $\epsilon >0$. All but finitely many solutions $(m,n)$ of the inequality
$$\displaystyle \sum_{v \in M_{k}\setminus S_{0}}-\log^{-}\max\{|F(m)|_{v},|G(n)|_{v}\}>\epsilon \max\{m,n\}$$
satisfy one of finitely many linear relations
$$(m,n)=(a_{i}t+b_{i},c_{i}t+d_{i}),\ t \in \mathbb{Z},\ i=1,\ldots, r,$$
where $a_{i},b_{i},c_{i},d_{i} \in \mathbb{Z},\ a_{i}c_{i}\neq 0$, and the linear recurrences $F(a_{i}\bullet+b_{i})$ and $G(c_{i}\bullet+d_{i})$ have a nontrivial common factor for $i=1,\ldots,r.$
\end{Th}

Grieve and Wang \cite{GW} have extended Theorem \ref{lrsle} to general linear recurrence sequences. 
\begin{Th}[Grieve, Wang \cite{GW}]\label{GW}
Let
$$\displaystyle F(m)=\sum_{i=1}^{s}p_{i}(m)\alpha_{i}^{m},$$
$$\displaystyle G(n)=\sum_{j=1}^{t}q_{j}(n)\beta_{j}^{n},$$
for $n \in \mathbb{N}$, be algebraic linear recurrence sequences, defined over a number field $k$, such that their roots generate together a torsion-free multiplicative subgroup $\Gamma$ of $k^{\times}$. Suppose that 
$$\displaystyle \max_{i,j} \{|\alpha_{i}|_{v},|\beta_{j}|_{v}\} \geq 1,$$
for any $v \in M_{k}$. Let $\epsilon >0$ and consider the inequality
\begin{align}\label{GWineq}
    \log \gcd(F(n),G(n)) < \epsilon\max\{m,n\}
\end{align}
for pairs of positive integers $(m,n) \in \mathbb{N}^{2}$. The following two assertions hold true.
\begin{enumerate}
\item Consider the case that $m=n$. If the inequality (\ref{GWineq}) is valid for infinitely many positive integers $(n,n) \in \mathbb{N}^{2}$, then $F$ and $G$ have a non-trivial common factor.
\item Consider the case that $m \neq n$. If the inequality (\ref{GWineq}) is valid for infinitely many pairs of positive integers $(m,n) \in \mathbb{N}^{2}$, with $m \neq n$, then the roots of $F$ and $G$ are multiplicatively dependent (see Def \ref{indroot}). Further, in this case, there exist finitely many pairs of integers $(a,b) \in \mathbb{Z}^{2}$ such that
$$|ma+nb|=o(\max\{m,n\}).$$
\end{enumerate}
\end{Th}
The proof of Theorem \ref{GW} in \cite{GW} is based on a ``moving targets" version of Theorem \ref{gcdle}. We will give an alternative proof of Theorem \ref{GW} and also give a quantitative improvement in which the error term $o(\max\{m,n\})$ can be controlled as a constant multiple of $\log\max\{m,n\}$.\\

In Section 3 and 4, we will give the proofs of the main Diophantine approximation results and the application to linear recurrence sequences, respectively.









\section{Preliminaries}
\subsection{Absolute values and heights}
Let $k$ be a number field, $M_{k}$ the set of places of $k$ and $\mathcal{O}_{k}$ the ring of integers of $k$. For $v \in M_{k}$, let $k_{v}$ denote the completion of $k$ with respect to $v$. Throughout the paper, we normalize the absolute value $|\cdot |_{v}$ corresponding to $v\in M_{k}$ as follows: If $v$ is archimedean and $\sigma$ is the corresponding embedding $\sigma:k \to \mathbb{C}$, then for $x \in k^{*}$, $|x|_{v}=|\sigma(x)|^{|k_{v}:\mathbb{R}|/|k:\mathbb{Q}|}$; if $v$ is non-archimedean corresponding to a prime ideal $\mathscr{P}$ in $\mathcal{O}_{k}$ which lies above a rational prime $p$, then it is normalized so that $|p|_{v}=p^{-|k_{v}:\mathbb{Q}_{p}|/|k:\mathbb{Q}|}$. In this notation, we have the product formula:
$$\displaystyle \prod_{v \in M_{k}}|x|_{v}=1$$
for all $x \in k^{*}$.

Let $S$ be a finite set of places in $M_{k}$. The ring of $S$-integers and the group of $S$-units are denoted by $\mathcal{O}_{k,S}$ and $\mathcal{O}_{k,S}^{*}$ respectively.\\

For a point $P=(\alpha_{0}:\alpha_{1}:\cdots:\alpha_{n}) \in \mathbb{P}^{n}(k)$, we define its height to be
$$\displaystyle h(P)=\sum_{v \in M_{k}}\log \max \{|\alpha_{0}|_{v},\ldots,|\alpha_{n}|_{v}\}$$
and for any $x \in k$, its height $h(x)$ is defined to be the height of the point $(1:x)$ in $\mathbb{P}^{1}(k)$.\\

We define local height functions as in \cite{BG}. Let $V$ be a projective variety over a number field $k$. Let $D$ be a Cartier divisor on $V$ and $v \in M_{k}$. First we define the support of a Cartier divisor $D=(U_{\alpha},f_{\alpha})_{\alpha \in I}$ to be 
$$\text{supp}(D):=\bigcup_{\alpha}\{x \in U_{\alpha}|f_{\alpha} \not \in \mathcal{O}_{V,x}^{*}\},$$
where $\mathcal{O}_{V,x}^{*}$ is the group of units in the local ring $\mathcal{O}_{V,x}$. We use $h_{D}:V(k) \to \mathbb{R}$ and $\lambda_{D,v}:\ V(k)\setminus \text{Supp}(D) \to \mathbb{R}$ to denote a height function associated to $D$ and local height function associated to $D$ and $v$ respectively, such that the local-global relation
$$\displaystyle \sum_{v \in M_{k}}\lambda_{D,v}(P)=h_{D}(P)+O(1)$$
is true of all points $P \in V(k)\setminus\text{Supp}(D)$. In particular, if $D$ is a hypersurface in $\mathbb{P}^{n}$ given by a homogeneous polynomial $F(x_{0},\ldots,x_{n})=0$ of degree $d$, we have a choice of local height function
$$\displaystyle \lambda_{D,v}(P)=\log \max_{i=0,\ldots,n} \dfrac{|\alpha_{i}|_{v}^{d}}{|F(P)|_{v}}=\log\dfrac{|P|_{v}^{d}}{|F(P)|_{v}}$$
where $P$ is written in coordinates $(\alpha_{0}:\cdots:\alpha_{n}) \in \mathbb{P}^{n}(k)\setminus \text{Supp}(D)$ and $|P|_{v}=\max_{i}|\alpha_{i}|_{v}$. For any $x\in k$ and $v\in M_{k}$, we define the local height of $x$ with respect to $v$ to be $\lambda_{v}(x)=\log\max\{1,|x|_v\}$.\\

For a point $P=(\alpha_{1},\ldots,\alpha_{n}) \in \mathbb{G}_{m}^{n}(k)$ and a place $v \in M_{k}$, we define its height and local height as
$$\displaystyle h(P)=\sum_{v \in M_{k}}\log \max \{1,|\alpha_{1}|_{v},\ldots,|\alpha_{n}|_{v}\}$$ 
and 
$$\lambda_{v}(P)=\log \max \{1,|\alpha_{1}|_{v},\ldots,|\alpha_{n}|_{v}\},$$ 
respectively.\\

A powerful tool in Diophantine Approximation is the famous Schmidt's Subspace Theorem, which will be the primary tool used in the proofs of this paper.
\begin{Th}[Schmidt's Subspace Theorem]
Let $k$ be a number field and $S \subset M_{k}$ a finite set of places, $n \in \mathbb{N}$ and $\epsilon > 0$. For every $v \in S$, let $\{L_{0}^{v},\ldots,L_{n}^{v}\}$ be a linearly independent set of linear forms in the variables $x_{0}, \ldots,x_{n}$ with coefficients in $k$. Then there are finitely many hyperplanes $T_{1},\ldots,T_{h}$ of $\mathbb{P}_{k}^{n}$ such that the set of solutions $\textbf{x}=(x_{0}:\cdots:x_{n}) \in \mathbb{P}_{k}^{n}(k)$ of
$$\displaystyle \sum_{v \in S}\log\prod_{i=0}^{n}\dfrac{|\textbf{x}|_{v}}{|L_{i}^{v}(\textbf{x})|_{v}}\geq (n+1+\epsilon)h(\textbf{x})+O(1)$$
is contained in $T_{1}\cup\cdots\cup T_{h}$. If we take $D_{v}$ to be the sum of divisors defined by $L_{i}^{v}$, $i=0,
\ldots, n$ and let $K_{\mathbb{P}^{n}}$ be the canonical divisor of $\mathbb{P}^{n}$, then this inequality can be written as
$$\displaystyle \sum_{v \in S}\lambda_{D_{v},v}(\textbf{x})+h_{K_{\mathbb{P}^{n}}}(\textbf{x}) \geq \epsilon h(\textbf{x})+O(1).$$
\end{Th}
If $H_{i}^{v}$ is the hyperplane defined by $L_{i}^{v}$, then the left-hand side of the inequality may be written as $\displaystyle \sum_{v \in M_{k}} \sum_{i=0}^{n} \lambda_{H_{i}^{v},v}(x)$ up to $O(1)$.

\subsection{Generalized Greatest Common Divisors}\label{pregcd}
One can extend the notion of $\log \gcd (a,b)$ to all algebraic numbers. Note that for $a$ and $b$ integers, we calculate their greatest common divisor as:
$$\displaystyle \begin{aligned}
\log \gcd (a,b)&=\sum_{p \  \text{prime}}\min\{\text{ord}_{p}(a),\text{ord}_{p}(b)\}\log p\\
&=-\sum_{v \in M_{\mathbb{Q},\text{fin}}}\log \max \{|a|_{v},|b|_{v}\}\\
&=-\sum_{v \in M_{\mathbb{Q},\text{fin}}}\log^{-} \max\{|a|_{v},|b|_{v}\}
\end{aligned}$$
where $M_{\mathbb{Q},\text{fin}}$ is the set of nonarchimedean places of $\mathbb{Q}$ and $\log^{-}z=\min\{0,\log z\}$.  Similarly we define  $\log^{+}z=\max\{0,\log z\}$. With this observation, by adding contributions of archimedean places, the generalized greatest common divisor is defined as:
\begin{Def}
Let $a, b \in \bar{\mathbb{Q}}$ be two algebraic numbers, not both zero. We define the generalized logarithmic greatest common divisors of $a$ and $b$ by
$$\displaystyle \log \gcd (a,b)=-\sum_{v \in M_{k}}\log^{-} \max\{|a|_{v},|b|_{v}\}$$
where $k$ is any number field containing both $a$ and $b$. 
\end{Def}

\subsection{Linear recurrence sequences}
Here we give some basic definitions and results involving linear recurrence sequences.\\

\begin{Def}\label{lrs}
A linear recurrence is a sequence $a=(a(i))$ of complex numbers satisfying a homogeneous linear recurrence relation
$$a(i+n)=s_{1}a(i+n-1)+\cdots+s_{n-1}a(i+1)+s_{n}a(i),\; i \in \mathbb{N}$$
with constant coefficients $s_{j} \in \mathbb{C}$. 
\end{Def}

\begin{Def}
The polynomial 
$$f(X)=X^{n}-s_{1}X^{n-1}-\cdots -s_{n-1}X-s_{n}$$
associated to the relation in Definition $\ref{lrs}$ is called its characteristic polynomial and the roots of this polynomial are said to be its roots. 
\end{Def}

\begin{Def}\label{gps}
A generalized power sum is a finite polynomial-exponential sum
$$\displaystyle a(i)=\sum_{j=1}^{m}A_{j}(i)\alpha_{j}^{i},\; i \in \mathbb{N}$$
with polynomial coefficients $A_{j}(z) \in \mathbb{C}[z]$. The $\alpha_{j}$ are the roots of the sequence $a(i)$.
\end{Def}

It is a well-known fact that every linear recurrence sequence $a(x)$ can be written in the form of a generalized power sum and in fact these two forms are equivalent, see \cite{EPSW}. Throughout this paper, linear recurrence sequences are presented in the form of a generalized power sum.\\

The linear recurrence sequence $a(i)$ is called degenerate if it has a pair of distinct roots whose ratio is a root of unity. Otherwise, it is called non-degenerate.\\

Fix a number field $k$. Let us define two linear recurrence sequences $F(n)$ and $G(n)$ by generalized power sums
$$\displaystyle F(n)=\sum_{i=1}^{m}A_{i}(n)\alpha_{i}^{n}$$
$$\displaystyle G(n)=\sum_{i=1}^{l}B_{i}(n)\beta_{i}^{n}$$
where $A(n)$, $B(n)$ are polynomials over $k$ and $\alpha_{i}$ and $\beta_{i}$ are roots in $k^{*}$. Let $\Gamma$ be the multiplicative group generated by all $\alpha_{i}$ and $\beta_{i}$ with a set of generators $\{u_{1},\ldots,u_{r}\}$. Then we can write $F(n)$ and $G(n)$ as
$$F(n)=f(n,u_{1}^{n},\ldots,u_{r}^{n})$$
$$G(n)=g(n,u_{1}^{n},\ldots,u_{r}^{n})$$
where $f$ and $g$ are rational functions in $x_{0},\ldots,x_{r}$ of the form:
$$f(x_{0},\ldots,x_{r})=\dfrac{\tilde{f}(x_{0},\ldots,x_{r})}{x_{1}^{a_{1}}\cdots x_{r}^{a_{r}}}$$
$$g(x_{0},\ldots,x_{r})=\dfrac{\tilde{g}(x_{0},\ldots,x_{r})}{x_{1}^{b_{1}}\cdots x_{r}^{b_{r}}}$$
with $\tilde{f}$, $\tilde{g}$ polynomials, i.e., $f,g\in k[x_{0},x_{0}^{-1},\ldots, x_{r},x_{r}^{-1}]$ are Laurent polynomials. In particular, the ring of such Laurent polynomials is a localization of $k[x_{0},\ldots,x_{r}]$, so it is a UFD.\\

It is obvious that linear recurrence sequences are closed under term-wise sum and product from the generalized power sum point of view, hence we can talk about the sum and product of two recurrence sequences. Let $\mathcal{H}_{\Gamma}(k)$ be the ring of linear recurrence sequences whose coefficient polynomials are over $k$ and roots belonging to a torsion-free multiplicative group $\Gamma\subset k^{*}$. We say $F(n),G(n) \in \mathcal{H}_{\Gamma}(k)$ are coprime if there does not exist a non-unit $H(n) \in \mathcal{H}_{\Gamma}(k)$ such that $F(n)=H(n)F_{0}(n)$ and $G(n)=H(n)G_{0}(n)$ with $F_{0}(n),G_{0}(n) \in \mathcal{H}_{\Gamma}(k)$. Recall that for $F(n),G(n )\in \mathcal{H}_{\Gamma}(k)$ and a choice of generators of the torsion-free group $\Gamma$, there are associated Laurent polynomials $f$ and $g$ respectively; if two such recurrence sequences are coprime then the two associated Laurent polynomials are also coprime.\\

We also need a well-known theorem on the structure of the zeros of a linear recurrence:\\
\begin{Th}[Skolem-Mahler-Lech]\label{SML}
The set of indices of the zeros of a linear recurrence sequence comprises a finite set together with a finite number of arithmetic progressions. If the linear recurrence sequence is nondegenerate, then there are only finitely many zeros.
\end{Th}

\subsection{Almost $S$-units and almost $S$-unit equations}
The definition of almost $S,\delta$-units was already given as in Definition \ref{asu}. Here are some remarks about this definition and its properties.

\begin{Rem}
Silverman has defined ``quasi-$S$-integers" in \cite{SilSInt}. For a number field $k$, a finite set of places $S$ and $\epsilon >0$, the set of quasi-$S$-integers are defined as
$$R_{S}(\epsilon):=\{x\in k: \sum_{v \in S}\max\{|x|_{v},0\}\geq \epsilon h(x)\}.$$
Silverman's notion of quasi-$S$-integers can be compared with our notion of almost $(S,\delta)$-units as follows: if $x \in k_{S,1-\epsilon}$ then $x\in R_{S}(\epsilon)$, and if $x\in R_{S}(\epsilon)$ then $x\in k_{S,2-\epsilon}$.
\end{Rem}

\begin{Rem}
In Evertse's work \cite{Ev2}, he defined, for constants $c,d$ with $c>0, d \geq 0$, a $(c,d,S)$-admissible point $P=(x_0:\ldots:x_n) \in \mathbb{P}(k)$ should satisfy:
\begin{enumerate}
    \item all $x_i$ can be chosen as $S$-integers.
    \item $\displaystyle \prod_{v\in S}\prod_{i=0}^n |x_i|_v \leq c \cdot H(P)^d$.
\end{enumerate}
Clearly, this also generalizes the notions of $S$-units and $S$-integers. Indeed, $(1,0,S)$-admissible points can be chosen to be all $S$-units. In our definition, if a $n$-tuple $P$ is in $k_{S,\delta}^{n}$, we can find some suitable $c,d$ such that $P$ is $(c,d,S)$-admissible. But we shall note that the other implication is not true in general.
\end{Rem}

\begin{Rem}\label{gmn}
We note that $k_{S,\delta}^{n} \subset \mathbb{G}_{m}^{n}(k)_{S,n\delta}$ and when $\delta=0$ we recover $n$-tuples of $S$-units, $G_{m}^{n}(k)_{S,0}=(\mathcal{O}_{k,S}^*)^{n}$.
\end{Rem}

\begin{Rem}\label{stand}
We use projective height to define almost $S$-units in $\mathbb{G}_{m}^{n}(k)_{S,\delta}$. In other references standard height is frequently used, where for a point $P=(x_{1},\ldots,x_{n}) \in \mathbb{G}_{m}^{n}(k)_{S,\delta}$, 
$$\displaystyle h_{stand}(P):=\sum_{i=1}^{n}h(x_{n}).$$ 
Local heights are defined similarly as 
$$\displaystyle \lambda_{stand,v}(P):=\sum_{i=1}^{n}\lambda_{v}(x_{n}).$$ 
One can verify that if $P \in \mathbb{G}_{m}^{n}(k)$ is an $(S,\delta)$-unit under the projective height, then it is an $(S,n\delta)$-unit under the standard height. Indeed, in this case we have
$$\begin{aligned}
\sum_{v \not \in S}\lambda_{stand,v}(P)+\lambda_{stand,v}(1/P)&=\sum_{v \not \in S}(\sum_{i=1}^{n}\lambda_{v}(x_{i})+\lambda_{v}(1/x_{i}))\\
&\leq n\sum_{v \not \in S}\lambda_{v}(P)+\lambda_{v}(1/P)\\
&\leq n\delta h(P) \leq n\delta h_{stand}(P).
\end{aligned}$$
\end{Rem}

Before the main proof, we need a generalized version of the unit equation. Evertse proved \cite[Theorem 1]{Ev2} a finiteness result on unit equations with $c,d,S$-admissible points using Subspace theorem. Here we state the theorem under the definition of almost $(S,\delta)$-units as a special case of Evertse's theorem.

\begin{Lemma}\label{unieq}
Let $k$ be a number field and let $S$ be a finite set of places of $k$ containing all archimedean places. Let $0<\delta <1/((n+1)(n+2))$. Let $\chi$ be the set of solutions of \\ 
$$x_{0}+ \cdots +x_{n}=1,\ (x_{0}, \ldots, x_{n}) \in k_{S,\delta}^{n+1}, $$
such that no proper subsum of $x_{0}+ \cdots +x_{n}$ vanishes. Then $\chi$ is a finite set.\\
\end{Lemma}

\begin{proof}
See \cite{Ev2}.
\end{proof}

\begin{Cor}\label{unieqc}
Let $0<\delta <1/((n+1)(n+2))$. Let $\chi$ be the set of solutions of
$$x_{0}+ \cdots +x_{n}=1$$
such that $(x_{0}, \ldots, x_{n}) \in  k_{S,\delta}^{n+1}$. Then there is a finite set $\mathcal{F} \subset k^{*}$ such that every $\textbf{x} \in \chi$ has at least one coordinate in $\mathcal{F}$.\\
\end{Cor}

\begin{proof}

See \cite{Ev2}.

\end{proof}

Lemma \ref{unieq} and Corollary \ref{unieqc} together give the generalized unit equation for $k_{S,\delta}^{n}$, which allows us to obtain finiteness of solutions in several of the following theorems.

\section{Diophantine Approximation}
In this section, our main goal is to give the proof of Theorem \ref{combined}. \\

In the following we will use the notation $\textbf{u}$ and $\textbf{i}$ for $n$-tuples $(u_{1},\ldots,u_{n})$ and $(i_{1},\ldots,i_{n})$, respectively, with $|\textbf{i}|=i_{1}+\cdots+i_{n}$  and denote by $\textbf{u}^{\textbf{i}}$ the multi-variable monomial $u_{1}^{i_{1}}\cdots u_{n}^{i_{n}}$. Let $m$ be a positive integer. For a subset $T \subset k[x_{1},\ldots,x_{n}]$, we let
$$T_{m}=\{p \in T| \deg p \leq m\},$$
and
$$T_{[m]}=\{p \in T| p \text{ is homogeneous of degree } m\}.$$
For $f,g \in k[x_{1}, \ldots, x_{n}]$, we let
$$(f,g)_{(m)}=\{fp+gq|\deg fp ,\deg gq \leq m\},$$
where $\deg$ denotes the (total) degrees of the polynomials.\\

Before the proof, we need a combinatorial lemma.

\begin{Lemma}\label{comb}
Let $m$  be a positive integer. Let $I=\{\mathbf{i}=(i_{0},\ldots,i_{n})\}$ be the set of $(n+1)$-tuples in $\mathbb{N}^{n+1}$ with $i_{0}+\cdots+i_{n}=m$.  Then 
$$\sum_{\mathbf{i}\in I}\mathbf{i}=\dfrac{m\binom{n+m}{n}}{n+1}(1,\ldots,1)$$
where addition and scalar multiplication are coordinate-wise.
\end{Lemma}

We also need Lemma 2.1 from \cite{CLZ}.

\begin{Lemma}\label{CLZlemma}
Let $F_{1},F_{2} \in k[x_{0},\ldots,x_{n}]$ be coprime homogeneous polynomials of degrees $d_{1}$ and $d_{2}$, respectively. Let $B \subset k[x_{0},\ldots,x_{n}]_{[m]}$ be a set of monomials of degree $m$ whose images are linearly independent in $k[x_{0},\ldots,x_{n}]_{[m]}/(F_{1},F_{2})_{[m]}$. Then 
$$\begin{aligned}
\sum_{\mathbf{x}^{\mathbf{j}} \in B}\text{ord}_{x_{i}}\mathbf{x}^{\mathbf{j}}&\leq \binom{m+n}{n+1}-\binom{m+n-d_{1}}{n+1}-\binom{m+n-d_{2}}{n+1}+\binom{m+n-d_{1}-d_{2}}{n+1}\\
&\leq d_{1}d_{2}\binom{m+n-2}{n-1}
\end{aligned}$$
for $i=0,\ldots,n$.
\end{Lemma}

\begin{proof}
Let $S=k[x_{0},\ldots,x_{n}]$. For an $l \in \mathbb{N}$ and a graded module $M$ over $S$, let $d_{M}(l)=\dim_{k}M_{[l]}$. Let $I$ be an ideal generated by a homogeneous polynomial of degree $i$. By the well-known theory of Hilbert polynomials, $d_{S/I}(l)=d_{S}(l)-d_{S}(l-i)$. In this case, 
$$\begin{aligned}
\dim({S_{[l]}/(F_{1},F_{2})_{[l]}})&=d_{S/(F_{1})}(l)-d_{S/(F_{1})}(l-d_{2})\\
&=d_{S}(l)-d_{S}(l-d_{1})-(d_{S}(l-d_{2})-d_{S}(l-d_{1}-d_{2}))\\
&=\binom{l+n}{n}-\binom{l+n-d_{1}}{n}-\binom{l+n-d_{2}}{n}+\binom{l+n-d_{1}-d_{2}}{n}.
\end{aligned}$$
Let $i\in \{0,\ldots,n\}$, let $S'_{[l]}$ be the image of $x_{i}^{l}k[x_{0},\ldots,x_{n}]_{[m-l]}$ in $S_{[m]}/(F_{1},F_{2})_{[m]}$.
Notice that 
$$\sum_{\mathbf{x}^{\mathbf{j}} \in B}\text{ord}_{x_{i}}\mathbf{x}^{\mathbf{j}}\leq \sum_{j=1}^{m}j(\dim S'_{[j]}-\dim S'_{[j+1]})=\sum_{j=1}^{m}\dim S'_{[j]},$$
and that $\dim S'_{[l]} \leq \dim S_{[m-l]}/(F_{1},F_{2})_{[m-l]}$. Hence, we have
$$\sum_{\mathbf{x}^{\mathbf{j}} \in B}\text{ord}_{x_{i}}\mathbf{x}^{\mathbf{j}}\leq \sum_{j=0}^{m-1}\dim S_{[j]}/(F_{1},F_{2})_{[j]}.$$
Using Pascal's identity for binomial coefficients,
$$\begin{aligned}
\sum_{\mathbf{x}^{\mathbf{j}} \in B}\text{ord}_{x_{i}}\mathbf{x}^{\mathbf{j}}&\leq \binom{m+n}{n+1}-\binom{m+n-d_{1}}{n+1}-\binom{m+n-d_{2}}{n+1}+\binom{m+n-d_{1}-d_{2}}{n+1}\\
&\leq d_{1}d_{2}\binom{m+n-2}{n-1}.
\end{aligned}$$
\end{proof}

\begin{Th}\label{notinsgen}
Let $k$ be a number field and let $S$ be a finite set of places of $k$ containing the archimedean places. Let $f$, $g \in k[x_{1}, \ldots, x_{n}]$ be coprime polynomials. For all $0< \delta <1$, there exists a proper Zariski closed subset $Z$ of $\mathbb{G}_{m}^{n}$ such that
$$\displaystyle - \sum_{v \in M_{k} \setminus S}\log^{-}\max\{|f(u_{1},\ldots, u_{n})|_{v}, |g(u_{1},\ldots,u_{n})|_{v}\}< C \delta^{1/2}\sum_{1\leq i \leq n}h(u_{i})$$
for all $\textbf{u}=(u_{1},\ldots,u_{n})\in \mathbb{G}_{m}^{n}(k)_{S,\delta} \setminus Z$, where $C=2(n^{2}\deg f +n\deg g )$ is a constant.\\
\end{Th}

\begin{proof}
This proof is modeled on the proof of Theorem 3.2 of \cite{Le}. \\

Consider the ideal $(f,g) \subset k[x_{1},\ldots,x_{n}]$. We first assume that 
$$(f,g)_{(m)} \neq k[x_{1},\ldots,x_{n}]_{m}.$$
It follows that the $k$-vector space $V_{m}=k[x_{1},\ldots,x_{n}]_{m}/(f,g)_{(m)}$ is not trivial. Let $\textbf{u}=(u_{1},\ldots,u_{n}) \in \mathbb{G}_{m}^{n}(k)_{S,\delta}$. For $v \in S$, we construct a basis $B_{v}$ for $V_{m}$ as follows. Choose a monomial $\textbf{x}^{\textbf{i}_{1}}\in k[x_{1},\ldots,x_{n}]_{m}$ so that $|\textbf{u}^{\textbf{i}_{1}}|_{v}$ is minimal subject to the condition $\textbf{x}^{\textbf{i}_{1}} \notin (f,g)_{(m)}$. Suppose now that $\textbf{x}^{\textbf{i}_{1}},\ldots,\textbf{x}^{\textbf{i}_{j}}$ have been constructed and are linearly independent modulo $(f,g)_{(m)}$, but don't span $k[x_{1},\ldots,x_{n}]_{m}$ modulo $(f,g)_{(m)}$. Then we let $\textbf{x}^{\textbf{i}_{j+1}} \in k[x_{1},\ldots,x_{n}]_{m}$ be a monomial such that $|\textbf{u}^{\textbf{i}_{j+1}}|_{v}$ is minimal subject to the condition that $\textbf{x}^{\textbf{i}_{1}},\ldots,\textbf{x}^{\textbf{i}_{j+1}}$ are linearly independent modulo $(f,g)_{(m)}$. In this way, we construct a basis of $V_{m}$ with monomial representatives $\textbf{x}^{\textbf{i}_{1}},\ldots,\textbf{x}^{\textbf{i}_{N'}}$, where $N'=N'_{m}=\dim V_{m}$. Let $I_{v}=\{\textbf{i}_{1},\ldots,\textbf{i}_{N'}\}$. We also choose a basis $\phi_{1},\ldots,\phi_{N}$ of the vector space $(f,g)_{(m)}$, where $N=N_{m}=\dim (f,g)_{(m)}$. Now for $\textbf{i}$, $|\textbf{i}| \leq m$, we have that
$$\displaystyle \textbf{x}^{\textbf{i}}+\sum_{j=1}^{N'}c_{\textbf{i},j}\textbf{x}^{\textbf{i}_{j}} \in (f,g)_{(m)}$$
for some choice of coefficients $c_{\textbf{i},j} \in k$. Then for each such $\textbf{i}$ there is a linear form $L_{\textbf{i}}^{v}$ over $k$ such that
$$\displaystyle L_{\textbf{i}}^{v}(\phi_{1},\ldots,\phi_{N})=\textbf{x}^{\textbf{i}}+\sum_{j=1}^{N'}c_{\textbf{i},j}\textbf{x}^{\textbf{i}_{j}}.$$
Note that $\{L_{\textbf{i}}^{v}(\phi_{1},\ldots,\phi_{N}): |\textbf{i}| \leq m, \textbf{i} \notin I_{v}\}$ is a basis for $(f,g)_{(m)}$, and $\{L_{\textbf{i}}^{v}: |\textbf{i}| \leq m, \textbf{i} \notin I_{v}\}$ is a set of $N$ linearly independent forms in $N$ variables. Let
$$\displaystyle P=\phi(\textbf{u}):=(\phi_{1}(\textbf{u}),\ldots,\phi_{N}(\textbf{u})) \in k^{N}.$$
We may additionally assume that $\phi(\textbf{u})\neq 0$ (by enlarging the set $Z$). From the triangle inequality and the definition of $\textbf{x}^{\textbf{i}_{1}},\ldots,\textbf{x}^{\textbf{i}_{N'}}$, for any $\textbf{i}$ with $|\textbf{i}| \leq m$, $\textbf{i} \notin I_{v}$, we have the key inequality
$$\log |L_{\textbf{i}}^{v}(P)|_{v} \leq \log |\textbf{u}^{\textbf{i}}|_{v}+C_{v}$$
where the constant $C_{v}$ depends only on $v \in S$ and the set $\{\textbf{i}_{1},\ldots,\textbf{i}_{N'}\}$ (and not on $\textbf{u}$).\\

We will apply the Subspace Theorem with the choice of linear forms $L_{\textbf{i}}^{v}$, $|\textbf{i}| \leq m$, $\textbf{i}\not \in I_{v}$, for each $v \in S$. We want to estimate the sum
$$\displaystyle \sum_{v \in S}\sum_{|\textbf{i}| \leq m, \textbf{i}\notin I_{v}}\log \dfrac{|P|_{v}}{|L_{\textbf{i}}^{v}(P)|_{v}}.$$
Towards this end, we estimate the sums
$$\displaystyle -\sum_{v \in S}\sum_{|\textbf{i}| \leq m, \textbf{i}\notin I_{v}}\log|L_{\textbf{i}}^{v}(P)|_{v} \ \ \text{and}\ \ \sum_{v \in S}\sum_{|\textbf{i}| \leq m, \textbf{i}\notin I_{v}}\log|P|_{v}$$ 
separately.\\

We have
$$\displaystyle -\sum_{v \in S}\sum_{|\textbf{i}| \leq m, \textbf{i} \notin \textbf{I}_{v}}\log|L_{\textbf{i}}^{v}(P)|_{v} \geq -\sum_{v \in S}\sum_{|\textbf{i}| \leq m, \textbf{i} \notin \textbf{I}_{v}} \log|\textbf{u}^{\textbf{i}}|_{v}-CN$$
where $C=\displaystyle \sum_{v \in S}C_{v}$. By the product formula,\\
$$\displaystyle \sum_{v \in S}\log|\textbf{u}^{\textbf{i}}|_{v}+\sum_{v \in M_{k}\setminus S}\log|\textbf{u}^{\textbf{i}}|_{v}=\sum_{v \in M_{k}}\log|\textbf{u}^{\textbf{i}}|_{v}=0.$$
It follows that,
$$\begin{aligned}
\displaystyle -\sum_{v \in S}\sum_{|\textbf{i}| \leq m, \textbf{i} \notin \textbf{I}_{v}} \log|\textbf{u}^{\textbf{i}}|_{v} &= -\sum_{v \in S}\sum_{|\textbf{i}| \leq m} \log|\textbf{u}^{\textbf{i}}|_{v}+\sum_{v \in S}\sum_{\textbf{i} \in \textbf{I}_{v}}\log|\textbf{u}^{\textbf{i}}|_{v}\\
&=\sum_{v \in S}\sum_{\textbf{i} \in \textbf{I}_{v}}\log|\textbf{u}^{\textbf{i}}|_{v}+\sum_{v \in M_{k} \setminus S}\sum_{|\textbf{i}| \leq m} \log|\textbf{u}^{\textbf{i}}|_{v}.
\end{aligned}$$

Let $d_{1}=\deg f$ and $d_{2}=\deg g$. By Lemma \ref{CLZlemma}, we have
$$\begin{aligned}
\displaystyle -\sum_{v \in S}\sum_{\textbf{i} \in \textbf{I}_{v}} \log|\textbf{u}^{\textbf{i}}|_{v} \leq d_{1}d_{2}\binom{m+n-2}{n-1}\sum_{1\leq i\leq n}h(u_{i}),
\end{aligned}$$
we find that,
$$\begin{aligned}
\displaystyle -\sum_{v \in S}\sum_{|\textbf{i}| \leq m, \textbf{i} \notin \textbf{I}_{v}}\log|L_{\textbf{i}}^{v}(P)|_{v} \geq& -d_{1}d_{2}\binom{m+n-2}{n-1}\sum_{1\leq i \leq n}h(u_{i})-CN\\
&+\sum_{v \in M_{k} \setminus S}\sum_{|\textbf{i}| \leq m} \log|\textbf{u}^{\textbf{i}}|_{v}.
\end{aligned}$$
By Lemma \ref{comb},
$$\begin{aligned}
\displaystyle \sum_{v \in M_{k} \setminus S}\sum_{|\textbf{i}| \leq m} \log|\textbf{u}^{\textbf{i}}|_{v}&=\sum_{|\textbf{i}| \leq m}\sum_{v \in M_{k} \setminus S} \log|\textbf{u}^{\textbf{i}}|_{v} \\
&=\dfrac{m\binom{n+m}{n}}{n+1}\sum_{v\in M_{k}\setminus S}\sum_{1\leq i \leq n}\log|u_{i}|_{v}\\
&\geq -\dfrac{m\binom{n+m}{n}}{n+1}\sum_{v\in M_{k}\setminus S}\sum_{1 \leq i \leq n}\lambda_{v}\left(\dfrac{1}{u_{i}}\right).
\end{aligned}$$
So we estimate,
$$\begin{aligned}
\displaystyle  -\sum_{v \in S}\sum_{|\textbf{i}| \leq m, \textbf{i} \notin \textbf{I}_{v}}\log|L_{\textbf{i}}^{v}(P)|_{v}
&\geq -d_{1}d_{2}\binom{m+n-2}{n-1}\sum_{1\leq i \leq n}h(u_{i})\\
-&\dfrac{m\binom{n+m}{n}}{n+1}\sum_{v\in M_{k}\setminus S}\sum_{1\leq i \leq n}\lambda_{v}\left(\dfrac{1}{u_{i}}\right)-CN.
\end{aligned}$$

On the other hand,
$$\displaystyle \sum_{v \in S}\sum_{|\textbf{i}| \leq m, \textbf{i} \notin \textbf{I}_{v}}\log|P|_{v}=N\sum_{v \in S}\log|P|_{v}=N(h(P)-\sum_{v \in M_{k} \setminus S}\log|P|_{v}).$$

Now since $\phi_{i}=fp_{i}+gq_{i}$, $\deg fp_{i}, \deg gq_{i} \leq m$, we have for $v \in M_{k} \setminus S$,
$$\begin{aligned}
\log|\phi_{i}(\textbf{u})|_{v} &= \log |fp_{i}(\text{u})+gq_{i}(\textbf{u})|_{v}\\
&\leq \log \max\{|fp_{i}(\textbf{u})|_{v},|gq_{i}(\textbf{u})|_{v}\}+O_{v}(1)\\
&\leq \log^{-}\max\{|fp_{i}(\textbf{u})|_{v},|gq_{i}(\textbf{u})|_{v}\}+m\lambda_{v}(\textbf{u})+O_{v}(1)\\
&\leq \log^{-}\max\{|f(\textbf{u})|_{v},|g(\textbf{u})|_{v}\}+m\lambda_{v}(\textbf{u})+O_{v}(1),
\end{aligned}$$
where $O_{v}(1)=0$ for all but finitely many $v$.

Then for $v \in M_{k} \setminus S$,
$$\log|P|_{v} \leq \log^{-}\max\{|f(\textbf{u})|_{v},|g(\textbf{u})|_{v}\}+m\lambda_{v}(\textbf{u})+C_{v}.$$

Now we sum over $v \in M_{k} \setminus S$ to get:

$$\sum_{v \in M_{k} \setminus S}\log|P|_{v}\leq \sum_{v \in M_{k} \setminus S}\log^{-}\max \{|f(\textbf{u})|_{v},|g(\textbf{u})|_{v}\}+ m\sum_{v \in M_{k}\setminus S}\lambda_{v}(\textbf{u})+O(1).$$

Then we find the estimate:
$$\begin{aligned}
\displaystyle \sum_{v \in S}\sum_{|\textbf{i}| \leq m, \textbf{i} \notin \textbf{I}_{v}}\log|P|_{v} \geq& N(h(P)-\sum_{v \in M_{k} \setminus S}\log^{-}\max \{|f(\textbf{u})|_{v},|g(\textbf{u})|_{v}\}\\
&-m\sum_{v \in M_{k}\setminus S}\sum_{1\leq i\leq n} \lambda_{v}(u_{i}))+O(1).
\end{aligned}$$

One also has the easy estimate
$$\displaystyle h(P) \leq m h(\textbf{u})+O(1).$$

Schmidt's Subspace Theorem implies that there exists a finite union $Z$ of proper subspaces of $k^{N}$ such that
$$\displaystyle \sum_{v \in S}\sum_{|\textbf{i}| \leq m,\textbf{i} \notin \textbf{I}_{v}}\log \dfrac{|Q|_{v}}{|L_{\textbf{i}}^{v}(Q)|_{v}} \leq (N+1)h(Q)$$
for all $Q \in k^{N}\setminus Z$.\\

Using the above estimates, if $P=\phi(\textbf{u}) \notin Z$, we find that up to an $O(1)$,
$$\begin{aligned}\displaystyle &N\left(h(P)-\sum_{v \in M_{k} \setminus S}\log^{-}\max \{|f(\textbf{u})|_{v},|g(\textbf{u})|_{v}\}-m\sum_{v \in M_{k}\setminus S}\sum_{1\leq i \leq n}\lambda_{v}(u_{i})\right)\\
&-d_{1}d_{2}\binom{m+n-2}{n-1}\sum_{1 \leq i \leq n}h(u_{i})-\dfrac{m\binom{n+m}{n}}{n+1}\sum_{v \in M_{k}\setminus S}\sum_{1 \leq i\leq n}\lambda_{v}\left(\dfrac{1}{u_{i}}\right)\\
&\leq (N+1)h(P)+CN.
\end{aligned}$$
Applying the estimate for $h(P)$, combining terms, and dividing by $N$, we obtain up to an $O(1)$,
$$\displaystyle 
\begin{aligned}
&-\sum_{v \in M_{k} \setminus S}\log^{-}\max \{|f(\textbf{u})|_{v},|g(\textbf{u})|_{v}\}-m\sum_{v\in M_{k}\setminus S}\sum_{1 \leq i \leq n}\lambda_{v}(u_{i})\\
&-\dfrac{m\binom{n+m}{n}}{N(n+1)}\sum_{v \in M_{k}\setminus S}\sum_{1 \leq i \leq n}\lambda_{v}\left(\dfrac{1}{u_{i}}\right)
\leq \dfrac{m+d_{1}d_{2}\binom{m+n-2}{n-1}}{N}\sum_{1 \leq i \leq n}h(u_{i}).
\end{aligned}$$

Since $f$ and $g$ are coprime, the ideal $(f,g)$ defines a closed subset of $\mathbb{A}^{n}$ of codimension at least $2$. Without loss of generality, assume $d_{1}\geq d_{2}$. By Lemma \ref{CLZlemma}, we find that $N'=\binom{m+n}{n}-\binom{m+n-d_{1}}{n}-(\binom{m+n-d_{2}}{n}-\binom{m+n-d_{1}-d_{2}}{n})\leq d_{1}d_{2}\binom{m+n-2}{n-2}$ and that $N=\binom{m+n}{n}-N' \geq \binom{m+n}{n}-d_{1}d_{2}\binom{m+n-2}{n-2}$. We assume now  $m \geq d_{1}n$. Then we have the estimate 
$$\begin{aligned}
\left(\binom{m+n}{n}-d_{1}d_{2}\binom{m+n-2}{n-2}\right)\bigg/\binom{m+n}{n}&=1-\dfrac{d_{1}d_{2}n(n-1)}{(m+n)(m+n-1)}\\
&\geq 1-\dfrac{d_{1}d_{2}n(n-1)}{d_{1}^{2}n^{2}}\\
&\geq 1-\dfrac{n-1}{n}=\dfrac{1}{n}.
\end{aligned}$$
Therefore we have
$$\displaystyle 
\begin{aligned}-\sum_{v \in M_{k} \setminus S}\log^{-}\max \{|f(\textbf{u})|_{v},|g(\textbf{u})|_{v}\} &\leq \dfrac{m+d_{1}d_{2}\binom{m+n-2}{n-1}}{1/n\binom{m+n}{n}}\sum_{1 \leq i \leq n}h(u_{i})\\
&+m\sum_{v \in M_{k}\setminus S}\sum_{1\leq i \leq n}\lambda_{v}(u_{i})\\
&+\dfrac{m\binom{m+n}{n}/(n+1)}{1/n\binom{m+n}{n}}\sum_{v \in M_{k}\setminus S}\sum_{1\leq i \leq n}\lambda_{v}\left(\dfrac{1}{u_{i}}\right).\\
\end{aligned}$$
One shall notice that
$$\displaystyle \dfrac{m+d_{1}d_{2}\binom{m+n-2}{n-1}}{1/n\binom{m+n}{n}}\leq \dfrac{2d_{1}d_{2}n^{2}}{m+1},$$
and that
$$\displaystyle \dfrac{\dfrac{m\binom{m+n}{n}}{n+1}}{1/n\binom{m+n}{n}}\leq m.$$
By Remark \ref{stand}, the condition $\displaystyle \sum_{1\leq i\leq n}h_{\bar{S}}(u_{i}) \leq n\delta \sum_{1\leq i \leq n}h(u_{i})$ is satisfied and we get
$$\displaystyle
\begin{aligned}-\sum_{v \in M_{k} \setminus S}\log^{-}\max \{|f(\textbf{u})|_{v},|g(\textbf{u})|_{v}\} &\leq \left(\dfrac{2d_{1}d_{2}n^{2}}{m+1}+mn\delta \right)\sum_{1\leq i \leq n}h(u_{i}).
\end{aligned}
$$

Now letting $m=\left \lfloor \dfrac{2d_{1}n}{\delta^{1/2}}\right \rfloor$, it follows that
$$\displaystyle -\sum_{v \in M_{k} \setminus S}\log^{-}\max \{|f(\textbf{u})|_{v},|g(\textbf{u})|_{v}\} \leq 2(d_{1}n^{2}+d_{2}n)\delta^{1/2}\sum_{1\leq i \leq n}h(u_{i}).$$

We can see this choice of $m$ satisfies the conditions $m \geq d_{1}n$ and $m \geq \max\{d_{1},d_{2}\}$. Now letting $C(n,d_{1},d_{2})=2(d_{1}n^{2}+d_{2}n)$, we have
$$\displaystyle -\sum_{v \in M_{k} \setminus S}\log^{-}\max \{|f(\textbf{u})|_{v},|g(\textbf{u})|_{v}\} \leq C(n,d_{1},d_{2})\delta^{1/2}\sum_{1\leq i\leq n}h(u_{i})$$
as long as $\textbf{u}$ does not lie in the proper closed subset coming from the exceptional set in the application of the Subspace Theorem.\\

Finally, we note that the choice of linear forms in the application of Schmidt's Subspace Theorem depends not on $\textbf{u}$, but on the choice of the monomial bases $B_{v}$, $v \in S$. Since for fixed $m$ there are only finitely many monomials of degree at most $m$, and hence only finitely many choices for these bases, we see that for fixed $m$ the given argument leads to only finitely many applications of Schmidt's Subspace Theorem (over all choices of $\textbf{u}$). Therefore there exists a proper Zariski closed subset $Z$ of $\mathbb{G}_{m}^{n}$ such that the inequality is valid for all $\textbf{u}=(u_{1},\ldots,u_{n}) \in \mathbb{G}_{m}^{n}(k)_{S,\delta}\setminus Z$.\\

Now consider the case when $(f,g)_{(m)}=k[x_{1},\ldots,x_{n}]_{m}$. We can find polynomials $\tilde{f},\tilde{g} \in k[x_{1},\ldots,x_{n}]$ such that 
$$f\tilde{f}+g\tilde{g}=1$$
with $\deg \tilde{f},\deg \tilde{g} \leq m$. Hence, for any $v \in M_{k}$ and $\textbf{u} \in \mathbb{G}_{m}^{n}(s)_{S,\delta}$, we have
$$\begin{aligned}
1=|(f\tilde{f}+g\tilde{g})(\textbf{u})|_{v}&\leq \max\{|f(\textbf{u})|_{v}|\tilde{f}(\textbf{u})|_{v},|g(\textbf{u})|_{v}|\tilde{g}(\textbf{u})|_{v}\}\\
&\leq \max\{|f(\textbf{u})|_{v},|g(\textbf{u})|_{v}\}\max\{|\tilde{f}(\textbf{u})|_{v},|\tilde{g}(\textbf{u})|_{v}\}.\\
\end{aligned}$$
Then we have
$$\max\{|f(\textbf{u})|_{v},|g(\textbf{u})|_{v}\}\geq \min\{|1/\tilde{f}(\textbf{u})|_{v},|1/\tilde{g}(\textbf{u})|_{v}\}.$$
Applying $-\log^{-}$ on both sides and summing over $v \in M_{k} \setminus S$, it follows that 
$$\begin{aligned}
-\sum_{v \in M_{k}\setminus S}\log ^{-}\max\{|f(\textbf{u})|_{v},|g(\textbf{u})|_{v}\} &\leq -\sum_{v \in M_{k}\setminus S}\log ^{-}\min\{|1/\tilde{f}(\textbf{u})|_{v},|1/\tilde{g}(\textbf{u})|_{v}\}\\
& =-\sum_{v \in M_{k}\setminus S}\min\{\log|1/\tilde{f}(\textbf{u})|_{v},\log|1/\tilde{g}(\textbf{u})|_{v},0\}\\
&=\sum_{v \in M_{k}\setminus S}\max\{\log|\tilde{f}(\textbf{u})|_{v},\log|\tilde{g}(\textbf{u})|_{v},0\}.\\
\end{aligned}$$
Now since $\deg \tilde{f},\deg \tilde{g} \leq m$, together with $\displaystyle \sum_{1\leq i \leq n}h_{\bar{S}}(u_{i}) \leq \delta \sum_{1\leq i\leq n}h(u_{i})$ (by Remark \ref{stand}), we obtain
$$-\sum_{v \in M_{k}\setminus S}\log ^{-}\max\{|f(\textbf{u})|_{v},|g(\textbf{u})|_{v}\} \leq mn\delta\sum_{1\leq i \leq n}h(u_{i}),$$
which is an even better estimate than the one obtained in the proof of the first case.
\end{proof}

By letting $\delta=\dfrac{\epsilon^{2}}{4n^{2}(n^{2}\deg f+n\deg g)^{2}}$, we obtain an immediate result:
\begin{Cor}\label{notins}
Let $k$ be a number field and let $S$ be a finite set of places of $k$ containing the archimedean places. Let $f$, $g \in k[x_{1}, \ldots, x_{n}]$ be coprime polynomials. For all $\epsilon >0$, there exist $\delta >0$ and a proper Zariski closed subset $Z$ of $\mathbb{G}_{m}^{n}$ such that
$$\displaystyle - \sum_{v \in M_{k} \setminus S}\log^{-}\max\{|f(u_{1},\ldots, u_{n})|_{v}, |g(u_{1},\ldots,u_{n})|_{v}\}< \epsilon \max\{h(u_{1}),\ldots,h(u_{n})\}$$
for all $\textbf{u}=(u_{1},\ldots,u_{n})\in \mathbb{G}_{m}^{n}(k)_{S,\delta} \setminus Z$.\\
\end{Cor}

The next theorem allows us to control the $S$-part of the greatest common divisor in Theorem \ref{combined}.
\begin{Th}\label{insgen}
Let $k$ be a number field and let $S$ be a finite set of places of $k$ containing the archimedean places. Let $f \in k[x_{1}, \ldots, x_{n}]$ be a polynomial of degree $d$ that doesn't vanish at the origin $(0,\ldots,0)$. For all $0< \delta <1$, there exists a proper Zariski closed subset $Z$ of $\mathbb{G}_{m}^{n}$ such that
$$\displaystyle -\sum_{v \in S} \log^{-}|f(u_{1},\ldots,u_{n})|_{v} <4nd\delta\sum_{1\leq i \leq n}h(u_{i})$$
for all $\textbf{u}=(u_{1},\ldots,u_{n})\in \mathbb{G}_{m}^{n}(k)_{S,\delta} \setminus Z$.\\
\end{Th}

\begin{proof}
In the following proof we will not consider the points $\{(u_{1},\ldots,u_{n}) \in \mathbb{G}_{m}^{n}(k)_{S,\delta} : f(u_{1},\ldots,u_{n})=0\}$. Since this set can be covered by a proper Zariski closed subset, by taking it into the exceptional set, we can ignore such points.\\

For a subset $S'$ of $S$, let $R_{S'}$ consist of the set of points $(u_{1},\ldots, u_{n}) \in \mathbb{G}_{m}^{n}(k)_{S,\delta}$ such that
$$S'=\{v \in S:\log |f(u_{1},\ldots,u_{n})|_{v}<0\}.$$
Then for $(u_{1},\ldots,u_{n}) \in R_{S'}$,
$$\log^{-}|f(u_{1},\ldots,u_{n})|_{v}=\begin{cases}
\log|f(u_{1},\ldots,u_{n})|_{v}, & v \in S',\\
0, & v \in S \setminus S'.
\end{cases}$$

Let $d=\deg f$, $m \in \mathbb{N}$ and $\phi: \mathbb{P}^{n} \to \mathbb{P}^{N}$, $\phi=(\phi_{0},\ldots,\phi_{N})$, $N=\binom{n+md}{n}-1$, be the $md$-uple embedding of $\mathbb{P}^{n}$ given by the set of monomials of degree $md$ in $k[x_{0},\ldots,x_{n}]$. Let $F=x_{0}^{d}f(x_{1}/x_{0},\ldots,x_{n}/x_{0})$ be the homogenization of $f$ in $k[x_{0},\ldots,x_{n}]$. Let $V_{md}$ be the vector space of homogeneous polynomials of degree $md$, and let $\text{Mon}_{md}$ consist of the set of all monomials in $k[x_{0},\ldots,x_{n}]$ of degree $md$.\\

If $v \in S'$, we construct a basis for $V_{md}$ as follows. Let $k_{\textbf{i}}=\left\lfloor \dfrac{\text{ord}_{x_0}\textbf{x}^{\textbf{i}}}{d} \right\rfloor$ and define $B_{v}^{\textbf{i}}=\dfrac{\textbf{x}^{\textbf{i}}}{x_{0}^{k_{\textbf{i}}d}}F^{k_{\textbf{i}}}$. Let $B_{v}$ be the set of all $B_{v}^{\textbf{i}}$. Since $f$ doesn't vanish at the origin, $x_{0}^{d}$ appears with a nonzero coefficient in $F$, and thus it's clear that $B_{v}$ is a basis for $V_{md}$.\\

If $v \in S \setminus S'$, then we let $B_{v}=\text{Mon}_{md}$. Applying the Subspace Theorem on $\mathbb{P}^{N}$ with appropriate linear forms, we find that for a fixed $\epsilon >0$
$$\displaystyle \sum_{v \in S}\sum_{Q \in B_{v}}\log \dfrac{|\phi(P)|_{v}}{|Q(P)|_{v}} \leq (N+1+\epsilon)h(\phi(P))$$
for all $P \in \mathbb{P}^{n}(k) \setminus Z$, where $Z=\phi^{-1}(Z')$ and $Z'$ is a finite union of hyperplanes in $\mathbb{P}^{N}$. From the definition of $B_{v}$, we can rewrite the left-hand side of above as
$$\displaystyle \sum_{v \in S}\sum_{Q \in \text{Mon}_{md}}\log \dfrac{|\phi(P)|_{v}}{|Q(P)|_{v}}-\sum_{\mathbf{i}}\sum_{v \in S'}\log \dfrac{|B_{v}^{\textbf{i}}(P)|_{v}}{|\textbf{x}^{\textbf{i}}(P)|_{v}} \leq (N+1+\epsilon)h(\phi(P)).$$
Suppose now that $(u_{1},\ldots,u_{n}) \in R_{S'}$ and let $P=[1:u_{1}:\ldots:u_{n}] \in \mathbb{P}^{n}(k)$. It follows immediately that for $B_{v}^{\textbf{i}}$ with $k_{\textbf{i}}d \leq \text{ord}_{x_{0}}\textbf{x}^{\textbf{i}}< (k_{\textbf{i}}+1)d$,
$$\displaystyle -\sum_{v \in S'}\log \dfrac{|B_{v}^{\textbf{i}}(P)|_{v}}{|\textbf{x}^{\textbf{i}}(P)|_{v}}=-k_{\textbf{i}}\sum_{v \in S'}\log |f(u_{1},\ldots,u_{n})|_{v} .$$
Letting $I =\sum_{\textbf{i}}k_{\textbf{i}}$,
$$\begin{aligned}
\displaystyle -\sum_{\textbf{i}}\sum_{v \in S'}\log \dfrac{|B_{v}^{\textbf{i}}(P)|_{v}}{|\textbf{x}^{\textbf{i}}(P)|_{v}}&=-I\sum_{v \in S'}\log |f(u_{1},\ldots,u_{n})|_{v}\\
&=-I \sum_{v \in S}\log^{-}|f(u_{1},\ldots,u_{n})|_{v}.
\end{aligned}$$
By an easy calculation, we find that
$$\begin{aligned}
\displaystyle I&=m+\left(\binom{n+d}{n}-1\right)(m-1)+\cdots+\left(\binom{n+(m-1)d}{n}-\binom{n+(m-2)d}{n}\right)\\
&=1+\binom{n+d}{n}+\cdots+\binom{n+(m-1)d}{n}\\
&\geq \binom{n+(m-1)d}{n}.
\end{aligned}$$

Note that $\phi$ induces a natural map $\mathbb{G}_{m}^{n} \to \mathbb{G}_{m}^{N}$ and $\phi(\mathbb{G}_{m}^{n}(k)_{S,\delta}) \subset \mathbb{G}_{m}^{N}(k)_{S,\delta}$. Indeed,
$$\displaystyle \sum_{v \in M_{k} \setminus S}\log |\phi(P)|_{v}=\sum_{v \in M_{k}\setminus S}\log \max_{Q \in \text{Mon}_{md}} \{|Q(P)|_{v}\} \leq md\sum_{v \in M_{k}\setminus S}\log \max_{i}|u_{i}|_{v}.$$
Similarly, we have
$$\displaystyle \sum_{v \in M_{k}\setminus S}\log  \left|\dfrac{1}{\phi(P)}\right|_{v} \leq md\sum_{v \in M_{k}\setminus S}\log \max_{i} \left|\dfrac{1}{u_{i}}\right|_{v}.$$
Thus we have
$$\displaystyle \begin{aligned}
\sum_{v \in M_{k}\setminus S}\lambda_{v}(\phi(P))+\lambda_{v}\left(\dfrac{1}{\phi(P)}\right) &\leq \sum_{v \in M_{k}\setminus S}md\left(\lambda_{v}(P)+\lambda_{v}\left(\dfrac{1}{P}\right)\right)\\ 
&\leq md\delta h(P)\leq \delta h(\phi(P)).
\end{aligned}$$
Now since $\phi(P) \in \mathbb{G}_{m}^{N}(k)_{S,\delta}$ and $\displaystyle \min_{i}|\phi_{i}(P)|_v \leq |Q(P)|_v $ for all $Q\in \text{Mon}_{md}$, then
$$\begin{aligned}
\displaystyle &\sum_{v \in S}\sum_{Q \in Mon_{md}}\log \dfrac{|\phi(P)|_{v}}{|Q(P)|_{v}}\\
&=\sum_{v \in M_{k}}\sum_{Q \in Mon_{md}}\log \dfrac{|\phi(P)|_{v}}{|Q(P)|_{v}}-\sum_{v \in M_{k}\setminus S}\sum_{Q \in Mon_{md}}\log \dfrac{|\phi(P)|_{v}}{|Q(P)|_{v}}\\
&=(N+1)h(\phi(P)) - \sum_{v \in M_{k}\setminus S}\sum_{Q \in Mon_{md}}\log \dfrac{|\phi(P)|_{v}}{|Q(P)|_{v}}\\
&\geq(N+1)h(\phi(P))-(N+1)\left(\sum_{v \in M_{k}\setminus S}\log |\phi(P)|_{v}+\sum_{v \in M_{k}\setminus S}\log \left|\dfrac{1}{\phi(P)}\right|_{v}\right)\\
&\geq(N+1)(1-\delta)h(\phi(P)).
\end{aligned}$$\\
Therefore, we have
$$\displaystyle (N+1)(1-\delta)h(\phi(P))-I\sum_{v \in S}\log^{-}|f(u_{1},\ldots,u_{n})|_{v} \leq (N+1+\epsilon)h(\phi(P))$$
for all $(u_{1},\ldots,u_{n}) \in R_{S'}$ outside of some proper Zariski closed subset $Z$.
It follows that for a sufficiently small $\epsilon$
$$\begin{aligned}\displaystyle -\sum_{v \in S}\log^{-}|f(u_{1},\ldots,u_{n})|_{v}&<  \dfrac{(N+1+\epsilon/\delta)}{I}\delta h(\phi(P))\\
&<\dfrac{\binom{n+md}{n}}{\binom{n+(m-1)d}{n}}\delta mdh(P)\\
&=\dfrac{(n+(m-1)d+1)\cdots(n+md)}{((m-1)d+1)\cdots(md)}\delta mdh(P)
\end{aligned}$$
for all $(u_{1},\ldots,u_{n}) \in R_{S'}$ outside of some proper Zariski closed subset $Z$.
Choosing $m=\left\lceil\dfrac{n-2^{1/d}+1}{d(2^{1/d}-1)}+1\right \rceil$, we have 
$$\displaystyle \dfrac{n+(m-1)d+1}{(m-1)d+1} \leq 2^{1/d}.$$
Hence, $\displaystyle \dfrac{(n+(m-1)d+1)\cdots(n+md)}{((m-1)d+1)\cdots(md)} \leq 2$. Also notice that $m \leq 2n$, we obtain
$$\displaystyle -\sum_{v \in S}\log^{-}|f(u_{1},\ldots,u_{n})|_{v} <2\delta mdh(P) <4nd\delta h(P) < 4nd\delta \sum_{1\leq i \leq n}h(u_{i})$$
for all $(u_{1},\ldots,u_{n}) \in R_{S'}$ outside of some proper Zariski closed subset $Z$. In fact, since there are only finitely many choices of the subset $S' \subset S$, we find that the inequality holds for all $P \in \mathbb{G}_{m}^{n}(k)_{S,\delta} \setminus Z$, for some proper closed subset $Z$.\\
\end{proof}

From Theorem \ref{insgen}, the immediate result combined with Theorem \ref{notinsgen} is 
\begin{Th}\label{gcdcomb}
Let $k$ be a number field and let $S$ be a finite set of places of $k$ containing the archimedean places. Let $f,g \in k[x_{1}, \ldots, x_{n}]$ be polynomials that don't both vanish at the origin $(0,\ldots,0)$. For all $0< \delta <1$, there exists a proper Zariski closed subset $Z$ of $\mathbb{G}_{m}^{n}$ such that
$$\displaystyle -\sum_{v \in M_{k}} \log^{-}\max\{|f(u_{1},\ldots,u_{n})|_{v},|g(u_{1},\ldots,u_{n})|_{v}\} < C\delta^{1/2}\sum_{1\leq i\leq n}h(u_{i})$$
for all $\textbf{u}=(u_{1},\ldots,u_{n})\in \mathbb{G}_{m}^{n}(k)_{S,\delta} \setminus Z$, where $C=6(\deg f+\deg g)n^{2}$ is a constant.\\
\end{Th}

\begin{proof}
With not loss of generality, assume $\deg f \leq \deg g$ and $g$ doesn't vanish at the origin. Then applying Theorem \ref{insgen} to $g$, on the right hand side we obtain 
$$4n\delta\deg g \sum_{1\leq i\leq n}h(u_{i})<4n(\deg g +\deg f)\delta\sum_{1\leq i\leq n}h(u_{i}).$$
Combining with the inequality from Theorem \ref{notinsgen} finishes the proof.
\end{proof}

Now we are ready to show the desired result (Theorem \ref{combined}):
\begin{Cor}\label{ins}
Let $k$ be a number field and $S$ a finite set of places of $k$ containing the archimedean places. Let $f,g \in k[x_{1},\ldots, x_{n}]$ be polynomials that don't both vanish at the origin $(0,\ldots,0)$. For all $\epsilon >0$, there exist a $\delta >0$ and a proper Zariski closed subset $Z\subset\mathbb{G}_m^n$ such that:
$$\displaystyle -\sum_{v \in M_{k}} \log^{-}\max\{|f(u_{1},\ldots,u_{n})|_{v},|g(u_{1},\ldots,u_{n})|_{v} \}< \epsilon \max_{i}h(u_{i})$$
for all $(u_{1},\ldots,u_{n}) \in \mathbb{G}_{m}^{n}(k)_{S,\delta} \setminus Z$.\\
\end{Cor}

\begin{proof}
By letting $\delta=\left(\dfrac{\epsilon}{6n^{3}(\deg f+\deg g)}\right)^{2}$, we obtain the inequality from Theorem \ref{gcdcomb}. Then we apply Theorem 5 in \cite{Ev}, $Z$ can be replaced by a (possibly infinite) union of positive-dimensional torus cosets.
\end{proof}

As discussed in the following remark, under a normal crossings assumption, a result of Silverman shows that Vojta’s conjecture predicts an improvement to Theorem \ref{gcdcomb}. \\

\begin{Rem}\label{Si}
From Theorem 2 in \cite{Si}, if we assume that Vojta's Conjecture is true, there is an improvement of the inequality as in Theorem \ref{notinsgen}. Let $k$ be a number field. Fix $\epsilon >0$. For $f$ and $g$ homogeneous coprime polynomials in $k[x_{0},\ldots,x_{n}]$ and $Y=\{f=g=0\}$ that intersects the coordinate hyperplanes transversally, there is a proper closed subset $Z$ such that we have for all $\textbf{x}\in \mathbb{P}^{n}(k) \setminus Z$,
$$\log\gcd (f(\textbf{x}),g(\textbf{x})) \leq \epsilon \max\{h(x_{0}),\ldots,h(x_{n})\}+\dfrac{1}{1+\gamma \epsilon}\sum_{1\leq i\leq n}h_{\bar{S}}(x_{i}),$$\\
where $\gamma$ is a positive constant. \\

Suppose $h_{\bar{S}}(\textbf{x}) \leq \delta h(\textbf{x})$. Using the estimate
$$\sum_{1\leq i\leq n}h_{\bar{S}}(x_{i}) \leq n h_{\bar{S}}(\textbf{x}),$$
we get
$$\log\gcd (f(\textbf{x}),g(\textbf{x})) \leq  \left(\epsilon+\dfrac{n\delta}{1+\gamma \epsilon} \right) \sum_{1\leq i\leq n}h(x_{i})\leq  \left(\epsilon+\dfrac{\delta}{1+\gamma \epsilon} \right) n\sum_{1\leq i\leq n}h(x_{i}).$$
Letting $\epsilon=\dfrac{-1+\sqrt{1+4\gamma \delta}}{2\gamma}$, we obtain a similar inequality as in Theorem \ref{notinsgen},
$$\begin{aligned}
\log\gcd (f(\textbf{x}),g(\textbf{x})) &\leq \left(\dfrac{-1+\sqrt{1+4\gamma \delta}}{2\gamma}+\dfrac{\delta}{1+\dfrac{-1+\sqrt{1+4\gamma \delta}}{2}}\right)n\sum_{1\leq i\leq n}h(x_{i})\\
&=(-1+\sqrt{1+4\gamma \delta})\left(\dfrac{1}{2\gamma}+\dfrac{2\delta}{4\gamma \delta}\right)n\sum_{1\leq i\leq n}h(x_{i})\\
&\leq\left(-1+1+\dfrac{4\gamma \delta}{2}\right)\dfrac{1}{\gamma}n\sum_{1\leq i\leq n}h(x_{i})\\
&=  2\delta n \sum_{1\leq i\leq n}h(x_{i}).
\end{aligned}$$
Thus, under a normal crossings assumption, Vojta’s conjecture predicts a linear dependence on $\delta$ in place of the square root dependence in Theorem \ref{gcdcomb} (note, however, that without a normal crossings assumption, the dependence on the degree of $f$ and $g$ in Theorem \ref{gcdcomb} is necessary, as can be seen by taking high powers of appropriate polynomials).
\end{Rem}

\begin{Ex}
In this example we show that the predicted linear dependence on $\delta$ is sharp (if true). Let $\mathbb{Q}$ be the field of rationals and $S=\{p,\infty\}$ be a finite set of places in $M_{\mathbb{Q}}$. Let $0 < \delta <1$. Let $x=p^{m},u=p^{n}$ for positive integers $m$ and $n$ such that $P:=(x,u(x+1))$ satisfies $1/2 \delta h(P) \leq h_{\bar{S}}(P) \leq \delta h(P)$. Let $x_{1},x_{2}$ be the coordinates in $\mathbb{G}_{m}^{2}$, then we take $f=x_{1}+1$, $g=x_{2}$. We make the estimate
$$
\begin{aligned}
\log \gcd (f(P),g(P))&=-\sum_{v\not \in S}\log^{-}\max\{|x+1|_{v},|u(x+1)|_{v}\}\\
&-\sum_{v \in S}\log^{-}\max\{|x+1|_{v},|u(x+1)|_{v}\}\\
&\geq \sum_{v \not \in S}\lambda_{v}\left(\dfrac{1}{x+1}\right)=h(x+1).
\end{aligned}$$
One shall also notice that for $x \in \mathcal{O}_{S,\mathbb{Q}}^{*}$, we have $h_{\bar{S}}(P)=h_{\bar{S}}(x+1)=h(x+1)$. Then
It follows that
$$\log \gcd (f(P),g(P)) \geq 1/2 \delta h(P).$$
It's easily seen that one may choose infinitely many appropriate $x$ and $u$ such that the set of resulting points $P$ forms a Zariski dense set in $\mathbb{G}_{m}^{2}$. Therefore the dependence on $\delta$ has to be at least linear.
\end{Ex}


\section{Linear Recurrence Sequences}
In this section, our main goal is to give the proof of Theorem \ref{recfin}, which requires Corollary \ref{notins} from the previous section.

\begin{Lemma}\label{rec1}
Let\\
\centerline{$\displaystyle F(n)=\sum_{i=0}^{s}p_{i}(n)\alpha_{i}^{n}$}\\
define a nondegenerate algebraic linear recurrence sequence. Let $|\cdot|$ be an absolute value on $\bar{\mathbb{Q}}$ such that $|\alpha _{i}| \geq 1$ for some $i$. Let $0<\epsilon<1$. Then
$$-\log|F(n)|<\epsilon n$$
for all but finitely many $n \in \mathbb{N}$.\\
\end{Lemma}

\begin{proof}
Let $k$ be a number field and $S$ a finite set of places of $k$ such that $p_{i}(x) \in k[x]$, $\alpha_{i} \in \mathcal{O}_{k,S}^{*}$, $i=0,\ldots,s$, and $|\cdot|$ restricted to $k$ is equivalent to $|\cdot|_{v}$ for some $v \in S$ (note that if $|\cdot|$ is trivial, the lemma is obvious). The $s=0$ case is trivial and so we may assume that $s>0$. By taking sufficiently large $n$, we can always assume that $p_{i}(n)$ don't vanish simultaneously. It suffices to prove that \\
\centerline{$-\log |F(n)|_{v} <\epsilon n$}\\
for all but finitely many $n \in \mathbb{N}$.\\

Let $H_{i}$ be the coordinate hyperplane in $\mathbb{P}^{s}$ defined by $x_{i}=0$, $i=0,\ldots,s$. Let $H_{s+1}$ be the hyperplane in $\mathbb{P}^{s}$ defined by $x_{0}+x_{1}+\cdots+x_{s}=0$. Note that the $s+2$ hyperplanes $H_{0},\ldots, H_{s+1}$ are in general position. Let
$$\begin{aligned}
P&=[\alpha_{0}:\cdots:\alpha_{s}] \in \mathbb{P}^{s}(k)\\
P_{n}&=[p_{0}(n)\alpha_{0}^{n}:\cdots:p_{s}(n)\alpha_{s}^{n}] \in \mathbb{P}^{s}(k),\ n \in \mathbb{N}\\
Q_{n}&=[p_{0}(n):\cdots:p_{s}(n)] \in \mathbb{P}^{s}(k),\ n \in \mathbb{N}.
\end{aligned}$$
Let $h=\max\{1,h(P)\}$. Then the Schmidt Subspace Theorem gives that for some finite union of hyperplanes $Z$ in $\mathbb{P}^{s}$,
\begin{align}\label{subsp}
\displaystyle \sum_{i=0}^{s+1}m_{H_{i},S}(P_{n}) < (s+1+\epsilon/(4h))h(P_{n}) 
\end{align}

\noindent for all points $P_{n} \in \mathbb{P}^{s}(k) \setminus Z$. In fact, since $F$ is nondegenerate, by the Skolem-Mahler-Lech theorem, only finitely many points $P_{n}$ belong to the given hyperplanes in $\mathbb{P}^{s}$, and thus the inequality holds for all but finitely many $n$. By taking $n$ to be sufficiently large, we can assume that $h(Q_{n}) \leq \delta h(P_{n})$ with $\delta \leq \dfrac{\epsilon}{4(s+1)h}$, so that we assume $P_{n} \in \mathbb{G}_{m}^{n}(k)_{S,\delta}$. Since $\alpha_{i} \in \mathcal{O}_{k,S}^{*}$ for all $i$, $m_{H_{i},S}(P_{n}) \geq (1-\delta)h(P_{n})$, $i=0,\ldots,s$. Note also that
$$h(P_{n}) \leq nh(P)+h(Q_{n}) \leq \dfrac{n}{1-\delta}h(P)$$
for all $n$ sufficiently large. Substituting in \eqref{subsp}, we have
$$m_{H_{s+1},S}(P_{n}) < (\epsilon/(4h)+(s+1)\delta)\dfrac{n}{1-\delta}h(P)\leq\dfrac{n\epsilon}{2(1-\delta)}.$$
Note that $\delta=\dfrac{\epsilon}{4(s+1)h}\leq \dfrac{1}{4(s+1)}\leq 1/2$, so we have $1-\delta >1/2$ and then
$$m_{H_{s+1}}(P_{n}) < \epsilon n.$$
Pick $\alpha_{j}$ with $|\alpha_{j}|_{v} \geq 1$. Then
$$\displaystyle \max_{i}\log |p_{i}(n)\alpha_{i}^{n}|_{v} \geq \log|p_{j}(n)|_{v}|\alpha_{j}|_{v}^{n} \geq \log |p_{j}(n)|_{v}.$$
To give $p_{j}(n)$ an estimate, we can take the inequality
$$\log |p_{j}(n)|_{v} \geq -h(p_{j}(n)).$$
Then it follows that
$$\log |p_{j}(n)|_{v} \geq-\deg p_{j}\log n+O(1).$$
It follows that
$$\lambda_{H_{s+1},v}(P_{n})=\log \dfrac{\max_{i}|p_{i}(n)\alpha_{i}^{n}|_{v}}{|\sum_{i=0}^{s}p_{i}(n)\alpha_{i}^{n}|_{v}} \geq -\log |F(n)|_{v} -C' \log n$$
for some constant $C'$. Together with $m_{H_{s+1},S}(P_{n}) \geq \lambda_{H_{s+1},v}(P_{n})+O(1)$, we have for all $\epsilon >0$,
$$-\log |F(n)|_{v}< \epsilon n+C' \log n+O(1).$$
It follows that for all sufficiently large $n$,
$$-\log |F(n)|_{v} < \epsilon n.$$
\end{proof}

Now we can state Theorem 1.8 (i) of Grieve-Wang \cite{GW} on the greatest common divisor between the terms of two linear recurrence sequences with the same index and give an alternative proof:
\begin{Th}\label{recbad1}
Let
$$\displaystyle F(m)=\sum_{i=1}^{s}p_{i}(m)\alpha_{i}^{m}$$
$$\displaystyle G(n)=\sum_{j=1}^{t}q_{j}(n)\beta_{j}^{n}$$
define two algebraic linear recurrence sequences, where $p_{i}$ and $q_{j}$ are polynomials. Let $k$ be a number field such that all coefficients of $p_{i}$ and $q_{j}$ and $\alpha_{i}$, $\beta_{j}$ are in $k$, for $i=1,\ldots,s$, $j=1,\ldots,t$. Let
$$S_{0}=\{v \in M_{k}:\max\{|\alpha_{1}|_{v},\ldots,|\alpha_{s}|_{v},|\beta_{1}|_{v},\ldots,|\beta_{t}|_{v}\} < 1\}.$$
Let $\epsilon >0$. Then all but finitely many solutions $l \in \mathbb{N}$ of the inequality
$$\displaystyle \sum_{v \in M_{k} \setminus S_{0}}-\log^{-}\max\{|F(l)|_{v},|G(l)|_{v}\} > \epsilon l$$
lie in one of finitely many nontrivial arithmetic subprogressions:
$$a_{i}t+b_{i},\ \ \ t \in \mathbb{N}, i=1, \ldots, r$$
where $a_{i},b_{i} \in \mathbb{N}, a_{i} \neq 0$, and the linear recurrences $F(a_{i}\bullet+b_{i})$ and $G(a_{i}\bullet+b_{i})$ have a nontrivial common factor for $i=1,\ldots,r$.
Furthermore, if $F$ and $G$ are coprime and their roots generate a torsion-free group, then there are only finitely many  solutions to the inequality above.
\end{Th}

\begin{proof}
We begin with a couple of convenient reductions. First, by considering finitely many arithmetic progressions in $l$, we may reduce to the case where the combined roots of $F$ and $G$ generate a torsion-free group $\Gamma$ of rank $r$ (in particular, both $F$ and $G$ are nondegenerate). Let $S \supset S_{0}$ be a finite set of places of $k$, containing the archimedean places, such that all coefficients of $p_{i}$ and $q_{j}$ and $\alpha_{i}$, $\beta_{j}$ are in $\mathcal{O}_{k,S}^{*}$ for all $i$ and $j$. \\
By Lemma \ref{rec1}, 
$$\displaystyle \sum_{v \in S \setminus S_{0}}-\log^{-}\max\{|F(l)|_{v},|G(l)|_{v}\} \leq \dfrac{\epsilon}{2} l$$
for all but finitely many $l \in \mathbb{N}$. Thus it suffices to prove the statement of the theorem with the inequality:
$$\displaystyle \sum_{v \in M_{k} \setminus S}-\log^{-}\max\{|F(l)|_{v},|G(l)|_{v}\} < \epsilon l.$$

Let $u_{1},\ldots,u_{r}$ be generators for $\Gamma$. Let $f,g \in k[l,x_{1},\ldots,x_{r},x_{1}^{-1},\ldots,x_{r}^{-1}]$ be the Laurent polynomials corresponding to $F$ and $G$. We may write
$$f(l,x_{1},\ldots,x_{r})=x_{1}^{i_{1}}\cdots x_{r}^{i_{r}}f_{0}(l,x_{1},\ldots,x_{r}),$$
$$g(l,x_{1},\ldots,x_{r})=x_{1}^{j_{1}}\cdots x_{r}^{j_{r}}g_{0}(l,x_{1},\ldots,x_{r})$$
where $i_{1},\ldots,i_{r},j_{1},\ldots,j_{r} \in \mathbb{Z}$ and $f_{0} \in k[l,x_{1},\ldots,x_{r}]$, $g_{0} \in k[l,x_{1},\ldots,x_{r}]$ with $x_{i} \nmid f_{0}g_{0}$, $i=1,\ldots,r$. Let $F_{0}$ and $G_{0}$ be the linear recurrence sequences corresponding to $f_0$ and $g_0$, respectively. Since $u_{1},\ldots,u_{r} \in \mathcal{O}_{k,S}^{*}$, it follows that\\
$$\displaystyle \sum_{v \in M_{k} \setminus S}-\log^{-}\max\{|F(l)|_{v},|G(l)|_{v}\} = \sum_{v \in M_{k} \setminus S}-\log^{-}\max\{|F_{0}(l)|_{v},|G_{0}(l)|_{v}\}.$$
Then it suffices to prove the statement of the theorem with $F$ and $G$ replaced by $F_{0}$ and $G_{0}$, respectively. Note that since $x_{1},\ldots, x_{r}$ are units in $k[x_{1},\ldots,x_{r},x_{1}^{-1},\ldots,x_{r}^{-1}]$, replacing $F$ and $G$ by $F_{0}$ and $G_{0}$ has no effect on coprimality statements. Thus, we now assume that $F$ and $G$ correspond to polynomials $f$ and $g$ in $k[l,x_{1},\ldots,x_{r}]$ .\\

Suppose now that $F$ and $G$ are coprime (equivalently, $f$ and $g$ are coprime). Let
$$P_{n}=(n,u_{1}^{n},\ldots,u_{r}^{n}).$$
Now for a fixed sufficiently small positive $\delta$ (coming from the proof of Corollary \ref{notins}), take $n$ to be sufficiently large such that $\displaystyle h(n) \leq \delta n\min_{i} h(u_{i})$, and so $P_{n} \in \mathbb{G}_{m}^{r+1}(k)_{S,\delta}$.\\
By Corollary \ref{notins},
$$\displaystyle \sum_{v \in M_{k} \setminus S}-\log^{-}\max\{|f(P_{n})|_{v},|g(P_{n})|_{v}\}<\epsilon \max\{h(u_{1}^{n}),\ldots,h(u_{r}^{n})\}$$
for all $P_{n} \in \mathbb{G}_{m}^{r+1}(k)_{S,\delta}$ outside a proper Zariski closed set $Z$. Noting that $f(P_{n})=F(n)$ and $g(P_{n})=G(n)$, and also that 
$$\max\{h(u_{1}^{n}),\ldots,h(u_{r}^{n})\}=n \max\{h(u_{1}),\ldots,h(u_{r})\},$$ 
after possibly shrinking $\epsilon$,we can write the above inequality as
$$\displaystyle \sum_{v \in M_{k} \setminus S}-\log^{-}\max\{|F(n)|_{v},|G(n)|_{v}\}<\epsilon n.$$
Cover the exceptional set $Z$ by a hypersurface defined by a polynomial 
$$\displaystyle Exc(x_{1},\ldots,x_{r+1})$$
in $k[x_{1},\ldots,x_{r+1}]$ such that if $P_{n} \in Z$ then $Exc(P_{n})=0$. We can view $Exc(P_{n})$ as terms of a linear recurrence sequence $E(n)$ with $E$ non-degenerate. By the Skolem-Mahler-Lech theorem, there are only finitely many zeros for $E$, which completes the proof.\\
\end{proof}







Here we deal with a special case when $m$ and $n$ are algebraically related:
\begin{Lemma}\label{sinvar}
Let
$$\displaystyle F(m)=\sum_{i=1}^{s}p_{i}(m)\alpha_{i}^{m}$$
$$\displaystyle G(n)=\sum_{j=1}^{t}q_{j}(n)\beta_{j}^{n}$$
be two linear recurrence sequences over a number field $k$ and $S$ be a finite set of places in $M_{k}$ containing archimedean places and $S_{0}$, where $S_{0}$ is defined as
$$S_{0}=\{v \in M_{k}:\max\{|\alpha_{1}|_{v},\ldots,|\alpha_{s}|_{v},|\beta_{1}|_{v},\ldots,|\beta_{t}|_{v}\} < 1\}.$$
Let $C \subset \mathbb{A}^{2}$ be an affine irreducible plane curve over $k$.  If there are infinitely many $(m,n) \in C(\mathbb{Z})$ satisfying the inequality
$$\displaystyle \sum_{v \in M_{k} \setminus S}-\log^{-}\max\{|F(m)|_{v},|G(n)|_{v}\} > \epsilon \max\{m,n\}$$
then $C$ is a line over $k$. In particular, if $m(t),n(t) \in \mathbb{Z}[t]$ are polynomials that are not linearly related, then the inequality 
$$\displaystyle \sum_{v \in M_{k} \setminus S}-\log^{-}\max\{|F(m(t))|_{v},|G(n(t))|_{v}\} > \epsilon \max\{m(t),n(t)\}$$
has only finitely many solutions $t \in \mathbb{Z}$.
\end{Lemma}

\begin{Rem}
Note that if $C$ is a line, the solutions are easily classified using Theorem \ref{recbad1}.
\end{Rem}

The following lemma is a basic fact from linear algebra, we state it without a proof.

\begin{Lemma}\label{rec2}
Let $\{v_{1},\ldots,v_{n}\}$ be a linearly independent subset of a normed vector space $X$. Then there exists a constant $c>0$ such that for every set of scalars $\{\alpha_{1},\ldots,\alpha_{n}\}$:
$$\|\alpha_{1}v_{1}+\cdots+\alpha_{n}v_{n}\| \geq c (|\alpha_{1}|+\cdots+|\alpha_{n}|).$$
\end{Lemma}

Let ${\rm Tor}(\overline{\mathbb{Q}}^*)$ denote the torsion subgroup of $\overline{\mathbb{Q}}^*$.  Since the height $h$ gives $\overline{\mathbb{Q}}^*/{\rm Tor(\overline{\mathbb{Q}}^*)}$ the structure of a normed vector space over $\mathbb{Q}$ as in Allcock and Vaaler \cite{AV}, we immediately find:

\begin{Lemma}
\label{multindlemma}
Let $u_{1},\ldots,u_{n}$ be multiplicatively independent elements of $\overline{\mathbb{Q}}^*$. Then there exists a constant $c>0$ such that for all $i_1,\ldots, i_n\in\mathbb{Z}$,
\begin{align*}
h(u_{1}^{i_{1}}\cdots u_{n}^{i_{n}}) \geq c \max_{j}{|i_{j}|}.
\end{align*}
\end{Lemma}

We now prove Lemma \ref{sinvar}.
\begin{proof}
Using the same reduction as in the proof of Theorem \ref{recbad1}, we can assume that the roots of $F$ and $G$ are $S$-units, and by considering finitely many congruence classes, we can assume that the roots of $F$ and $G$ generate a torsion free group.  Let $C$ be the affine curve defined by the algebraic relation $R(x_{1},x_{2})=0$, with $R(x_1,x_2)\in k[x_1,x_2]$ irreducible. If $C$ is not geometrically irreducible then $C(k)$ (and hence $C(\mathbb{Z})$) is finite, and so we further assume $C$ is geometrically irreducible. By Siegel's Theorem, $C(\mathbb{Z})$ is finite unless $C$ has genus 0 and $C$ has two or fewer distinct points at infinity, which we now assume. After replacing $k$ by a suitable finite extension, we can parametrize $C$ by Laurent polynomials $m(t), n(t)\in k[t,1/t]$.  Assume that $C$ is not a line, or equivalently, that $m(t)$ and $n(t)$ do not satisfy a linear relation.  

Let $\Gamma$ be the torsion free group generated by the roots of $F$ and $G$ and let $\{u_{1},\ldots, u_{r}\}$ be generators of $\Gamma$. Consider the points
$$P_{t}=(t,u_{1}^{m(t)},\ldots,u_{r}^{m(t)},u_{1}^{n(t)},\ldots,u_{r}^{n(t)}),$$
for $t\in k$ where, as we implicitly assume from now on, we have $m(t),n(t)\in \mathbb{Z}$. Then for some Laurent polynomials 
$$f(x_1,\ldots, x_{r+1}), g(x_1,x_{r+2},\ldots, x_{2r+1})\in k[x_1,\ldots, x_{2r+1}, x_1^{-1},\ldots, x_{2r+1}^{-1}],$$ 
we have $F(m(t))=f(P_t)$ and $G(n(t))=g(P_t)$. From the form of $f$ and $g$, we may write
\begin{align*}
f(x_1,\ldots, x_{r+1})&=x_{2}^{i_{1}}\cdots x_{r+1}^{i_{r}}c(x_1)\bar{f}(x_1,\ldots, x_{r+1}),\\
g(x_{1},x_{r+2},\ldots,x_{2r+1})&=x_{r+2}^{j_1}\cdots x_{2r+1}^{j_{r}}c(x_1)\bar{g}(x_{1},x_{r+2},\ldots,x_{2r+1})
\end{align*}
where $i_1,\ldots, i_r,j_1,\ldots, j_r\in\mathbb{Z}$, $\bar{f}$ and $\bar{g}$ are coprime polynomials in $k[x_1,\ldots, x_{2r+1}]$, and $c(x_1)$ is a Laurent polynomial in $x_1$.

By basic properties of heights, if $m(t),n(t)\in\mathbb{Z}$, then $h(t)\ll \log\max\{|m(t)|,|n(t)|\}$ and $h(P_t)\gg  \max\{|m(t)|,|n(t)|\}$. It follows that for any $\delta>0$, we have $P_t\in \mathbb{G}_{m}^{2r+1}(k)_{S,\delta}$ for all but finitely many $t\in k$ (with $m(t),n(t)\in \mathbb{Z}$).  Then Corollary \ref{notins} applies to $\bar{f}$ and $\bar{g}$ and we obtain that for any $\epsilon >0$ there exists a proper Zariski closed subset $Z\subset \mathbb{G}_{m}^{2r+1}$ such that
$$\displaystyle \sum_{v \in M_{k}\setminus S}-\log^{-}\max\{|\bar{f}(P_{t})|_{v},|\bar{g}(P_{t})|_{v}\}<\epsilon \max_{i=1,\ldots,r}\{h(u_{i}^{m(t)}), h(u_{i}^{n(t)})\}$$
for all points $P_{t}$ outside $Z$.  By elementary estimates, for all but finitely many $t\in k$,
\begin{align*}
\sum_{v\in M_k\setminus S}-\log^-|c(t)|_v\leq h(c(t))<\epsilon \max_{i=1,\ldots,r}\{h(u_{i}^{m(t)}), h(u_{i}^{n(t)})\}.
\end{align*}

Using this inequality and that $u_1,\ldots, u_r\in \mathcal{O}_{k,S}^*$, the inequality for $\bar{f}$ and $\bar{g}$ implies the inequality for $f$ and $g$:
\begin{align*}
\displaystyle \sum_{v \in M_{k}\setminus S}-\log^{-}\max\{|f(P_{t})|_{v},|g(P_{t})|_{v}\}<\epsilon \max_{i=1,\ldots,r}\{h(u_{i}^{m(t)}), h(u_{i}^{n(t)})\}
\end{align*}
for all points $P_{t}$ outside a proper Zariski closed subset $Z\subset \mathbb{G}_{m}^{2r+1}$. Set $(m,n)=(m(t),n(t))\in C(\mathbb{Z})$, and note that $f(P_{t})=F(m)$, $g(P_{t})=G(n)$, and 
$$\max\{h(u_{1}^{m}),\ldots,h(u_{r}^{m}),h(u_{1}^{n}),\ldots,h(u_{r}^{n})\}\leq\max \{m,n\} \max\{h(u_{1}),\ldots,h(u_{r})\}.$$ 
Then we can write the above inequality as
$$\displaystyle \sum_{v \in M_{k} \setminus S}-\log^{-}\max\{|F(m)|_{v},|G(n)|_{v}\}<\epsilon \max\{m,n\}.$$
It remains to show that there are only finitely many $t\in k$ with $m(t),n(t)\in \mathbb{Z}$ and $P_{t}\in Z$. Now we cover $Z$ by a hypersurface defined by an equation $z(x_1,\ldots, x_{2r+1})=0$.   Then every $P_{t}$ in $Z$ satisfies an equation

\begin{align*}
z(P_{t})=\sum_{w=1}^{K}P_{w}(t)u_{1}^{m(t)s_{1,w}}\cdots u_{r}^{m(t)s_{r,w}}u_{1}^{n(t)t_{1,w}}\cdots u_{r}^{n(t)t_{r,w}}=0,
\end{align*}
where $P_w\in k[t], w=1,\ldots, K$ are nonzero polynomials and the integer tuples $(s_{1,w},\ldots, s_{r,w},t_{1,w},\ldots, t_{r,w})$, $w=1,\ldots, K$, are distinct.  If $K=1$ then $t$ must be one of the finitely many roots of the polynomial $P_1(t)$. Otherwise, dividing by the first term we find
\begin{align}
\label{linrecuniteq}
\sum_{w=2}^{K}Q_{w}(t)u_{1}^{m(t)s'_{1,w}+n(t)t'_{1,w}}\cdots u_{r}^{m(t)s'_{r,w}+n(t)t'_{r,w}}=1,
\end{align}
where $Q_{w}(t), i=2,\ldots, K,$ are rational functions in $t$ and $s'_{i,w}=s_{i,w}-s_{i,1}$, $t'_{i,w}=t_{i,w}-t_{i,1}$.

Note that
\begin{align*}
h(Q_{w}(t))&=(\deg Q_w)h(t)+O(1)
\end{align*}
and by Lemma \ref{multindlemma} (assuming $m(t),n(t)\in \mathbb{Z}$ as usual)
\begin{align*}
h\left(u_{1}^{m(t)s'_{1,w}+n(t)t'_{1,w}}\cdots u_{r}^{m(t)s'_{r,w}+n(t)t'_{r,w}}\right)&\gg \max_i\{|m(t)s'_{i,w}+n(t)t'_{i,w}|\}\\
&=e^{\max_i h(m(t)s'_{i,w}+n(t)t'_{i,w})}\\
&\gg e^{h(t)\max_i \deg(ms'_{i,w}+nt'_{i,w})}\\
&\gg e^{h(t)}
\end{align*}
since $(s'_{i,w},t'_{i,w})\neq (0,0)$ for some $i$, and in this case $ms'_{i,w}+nt'_{i,w}$ must be nonconstant by our assumption that $m$ and $n$ aren't linearly related.

Since the terms in the sum in \eqref{linrecuniteq} are $S$-units outside the factors $Q_w(t)$, it follows from the height estimates above and the almost $S$-unit equation (Corollary \ref{unieqc}) that there exists a finite set $\mathcal{F}\subset k$ such that every solution $t\in k$ to \eqref{linrecuniteq} (with $m(t),n(t)\in\mathbb{Z}$) satisfies
$$Q_{w}(t)u_{1}^{m(t)s'_{1,w}+n(t)t'_{1,w}}\cdots u_{r}^{m(t)s'_{r,w}+n(t)t'_{r,w}}\in \mathcal{F}$$
for some $w$. By the height estimates above, 
$$h(Q_{w}(t)u_{1}^{m(t)s'_{1,w}+n(t)t'_{1,w}}\cdots u_{r}^{m(t)s'_{r,w}+n(t)t'_{r,w}})\gg e^{h(t)},$$
and Northcott's Theorem implies that there are only finitely many solutions $t\in k$ with $m(t),n(t)\in\mathbb{Z}$ satisfying \eqref{linrecuniteq}. It follows that there are only finitely many pairs $(m,n)\in C(\mathbb{Z})$ satisfying the inequality of the theorem.

\end{proof}

\begin{Def}\label{indroot}
Let $F$ and $G$ be two linear recurrence sequences. Suppose that the roots of $F$ and $G$ generate multiplicative torsion-free groups of rank $r$ and $s$, respectively. We say that the roots of $F$ and $G$ are multiplicatively independent if the combined roots generate a group of rank $r+s$. Otherwise, we say they are multiplicatively dependent. 
\end{Def}

The following result is a generalization of Theorem \ref{recbad1} under a multiplicative independence assumption, which was proved by Grieve-Wang \cite{GW}. Here we give an alternative proof:
\begin{Th}\label{recbad2}
Let
$$\displaystyle F(m)=\sum_{i=1}^{s}p_{i}(m)\alpha_{i}^{m}$$
$$\displaystyle G(n)=\sum_{j=1}^{t}q_{j}(n)\beta_{j}^{n}$$
define two algebraic linear recurrence sequences, where $p_{i}$ and $q_{j}$ are polynomials. Let $k$ be a number field such that all coefficients of $p_{i}$ and $q_{j}$ and $\alpha_{i}$, $\beta_{j}$ are in $k$, for $i=1,\ldots,s$, $j=1,\ldots,t$. Let
$$S_{0}=\{v \in M_{k}:\max\{|\alpha_{1}|_{v},\ldots,|\alpha_{s}|_{v},|\beta_{1}|_{v},\ldots,|\beta_{t}|_{v}\} < 1\}.$$
Let $\epsilon >0$. If we assume further the roots of $F$ and $G$ are independent, then all but finitely many $(m,n) \in \mathbb{N}^{2}$ satisfy the inequality
$$\displaystyle \sum_{v \in M_{k} \setminus S_{0}}-\log^{-}\max\{|F(m)|_{v},|G(n)|_{v}\} < \epsilon \max\{m,n\}.$$

In particular, if $S_{0}=\emptyset$, then all but finitely many $(m,n)$ satisfy the inequality
$$\log \gcd(F(m),G(n))<\epsilon \max\{m,n\}$$
\end{Th}

\begin{proof}
Notice that
$$\begin{aligned}
\sum_{v \in M_{k} \setminus S}-\log^{-}\max\{|F(m)|_{v},|G(n)|_{v}\}&\leq \min\{h(F(m)),h(G(n))\}\\
&\leq\mathcal{K}\min\{m,n\}
\end{aligned}$$
for some constant $\mathcal{K}$. Hence, for the inequality in the statement to be true, for a fixed $\epsilon > 0$,
$$\mathcal{K}\min\{m,n\} \geq \epsilon \max\{m,n\}.$$

The combined roots of $F$ and $G$ generate a torsion-free group $\Gamma$ of rank $r+s$ whose generators are $\{u_{1},\ldots, u_{r},v_{1},\ldots, v_{s}\}$ where $u_{1},\ldots,u_{r}$ generate the roots $\alpha_{i}$ and $v_{1},\ldots,v_{s}$ generate the roots $\beta_{j}$. By the same reduction step as in the previous proof, we assume all the coefficients of the polynomials $p_{i}$ and $q_{j}$ and all of the roots of $F$ and $G$ are $S$-units. We can also assume the Laurent polynomials $f$ and $g$ corresponding to $F$ and $G$ with respect to the roots $u_1,\ldots, u_r$ and $v_1,\ldots, v_s$, respectively, are polynomials.\\

Let $\hat{f},\hat{g} \in k[x_{1},\ldots,x_{r+s+2}]$ be polynomials such that
$$\hat{f}(x_{1},\ldots,x_{r+s+2})=f(x_{1},\ldots,x_{r+1})$$
$$\hat{g}(x_{1},\ldots,x_{r+s+2})=g(x_{r+2},\ldots,x_{r+s+2}).$$
Note that $\hat{f}$ and $\hat{g}$ are coprime since they involve disjoint sets of variables.\\

For $m,n \in \mathbb{N}$, let
$$P_{m,n}=(m,u_{1}^{m},\ldots,u_{r}^{m},n,v_{1}^{n},\ldots,v_{s}^{n}).$$
Let $\epsilon >0$ and let $\delta>0$ be the quantity from Corollary \ref{notins} for $\hat{f}, \hat{g}$, and $\epsilon$. After excluding finitely many pairs $(m,n)$, we can always assume that
$$\displaystyle h(m)+h(n)<\delta (r+s+2)\max_{i,j} \{h(u_{i}^{m}),h(v_{j}^{n})\}.$$
Therefore $P_{m,n} \in \mathbb{G}_{m}^{r+s+2}(k)_{S,\delta}$. Applying Corollary \ref{notins}, 
$$\displaystyle \sum_{v \in M_{k} \setminus S}-\log^{-}\max\{|\hat{f}(P_{m,n})|_{v},|\hat{g}(P_{m,n})|_{v}\}<\epsilon \max\{h(u_{1}^{m}),\ldots,h(u_{r}^{m}),h(v_{1}^{n}),\ldots,h(v_{s}^{n})\}$$
for all $P_{m,n} \in  \mathbb{G}_{m}^{r+s+2}(k)_{S,\delta}$ outside a proper Zariski closed set $Z \subset \mathbb{G}_{m}^{r+s+2}$. Noting that $\hat{f}(P_{m,n})=F(m)$, $\hat{g}(P_{m,n})=G(n)$, and
$$\displaystyle \max_{1\leq i \leq r, 1 \leq j \leq s}\{h(u_{i}^{m}),h(v_{j}^{n})\}\leq\max \{n,m\} \max_{1\leq i \leq r, 1 \leq j \leq s}\{h(u_{i}),h(v_{j})\},$$
we can write the above inequality as
$$\displaystyle \sum_{v \in M_{k} \setminus S}-\log^{-}\max\{|F(m)|_{v},|G(n)|_{v}\}<\epsilon \max\{n,m\}.$$
As in the $m=n$ case, we cover $Z$ by a hypersurface defined by a polynomial equation:
$$\displaystyle Exc(x_{1},\ldots,x_{r+s+2})=0.$$ 
Hence, all the points $P_{m,n}$ in $Z$ must satisfy the above equation. Therefore, if $P_{m,n} \in Z$, after combining the terms with the same exponents on $u_{1},\ldots,u_{r},v_{1},\ldots,v_{s}$, we obtain an equation in terms of $m,n,u_{1},\ldots, u_{r},v_{1},\ldots,v_{s}$:
$$\displaystyle Exc(m,u_{1}^{m},\ldots,u_{r}^{m},n,v_{1}^{n},\ldots,v_{s}^{n})=\sum_{w=1}^{W} P_{w}(m,n)u_{1}^{ms_{1,w}}\cdots u_{r}^{ms_{r,w}}v_{1}^{nt_{1,w}}\cdots v_{s}^{nt_{s,w}}=0,$$
where $P_{w}(m,n)$ is a non-zero polynomial in $m$ and $n$. It follows from Theorem \ref{recbad1} and Lemma \ref{sinvar} that after excluding finitely many pairs $(m,n)$ we can assume that $(m,n)$ is not a zero of any of the polynomials $P_w$.\\

Dividing both sides by the negative of the first term,
$$\displaystyle \sum_{w=2}^{W} \dfrac{P_{w}(m,n)u_{1}^{ms_{1,w}}\cdots u_{r}^{ms_{r,w}}v_{1}^{nt_{1,w}}\cdots v_{s}^{nt_{s,w}}}{{-P_{1}(m,n)u_{1}^{ms_{1,1}}\cdots u_{r}^{ms_{r,1}}v_{1}^{nt_{1,1}}\cdots v_{s}^{nt_{s,1}}}}=1.$$
Let $Q_{w}(m,n)=\dfrac{P_{w}(m,n)}{-P_{1}(m,n)}$ ($w=2,\ldots, W$), then
$$\displaystyle \sum_{w=2}^{W} Q_{w}(m,n)u_{1}^{m(s_{1,w}-s_{1,1})}\cdots u_{r}^{m(s_{r,w}-s_{r,1})}v_{1}^{n(t_{1,w}-t_{1,1})}\cdots v_{s}^{n(t_{s,w}-t_{s,1})}=1.$$
Letting $s'_{i,w}=s_{i,w}-s_{i,1}$, $t'_{i,w}=t_{i,w}-t_{i,1}$, we have
$$\displaystyle \sum_{w=2}^{W} Q_{w}(m,n)u_{1}^{ms'_{1,w}}\cdots u_{r}^{ms'_{r,w}}v_{1}^{nt'_{1,w}}\cdots v_{s}^{nt'_{s,w}}=1$$
with $s'_{i,w}$, $t'_{i,w}$ fixed and only depending on $Exc$.\\

As in the proof of Lemma \ref{sinvar}, it follows from Lemma \ref{multindlemma} that if $\min\{m,n\}$ is sufficiently large, then Corollary \ref{unieqc} applies to the equation
$$\displaystyle \sum_{w=2}^{W}Q_{w}(m,n)u_{1}^{ms'_{1,w}}\cdots u_{r}^{ms'_{r,w}}v_{1}^{nt'_{1,w}}\cdots v_{s}^{nt'_{s,w}}=1,$$
and we conclude that one of the summands on the left-hand side belongs to a finite set $\mathcal{F}$.  But since
$$h(Q_{w}(m,n)u_{1}^{ms'_{1,w}}\cdots u_{r}^{ms'_{r,w}}v_{1}^{nt'_{1,w}}\cdots v_{s}^{nt'_{s,w}}) \to \infty \text{ as } \min\{m,n\}\to\infty,$$
and $\min\{m,n\}\to \infty$ also means $\max\{m,n\}\to\infty$ by the remarks at the beginning of the proof, this implies that there are only finitely many possibilities for the pair $(m,n)$.
\end{proof}

We now prove a result in the general case where the roots of $F$ and $G$ are not necessarily independent. The following theorem gives an improvement to Theorem 1.8 (ii) of Grieve-Wang \cite{GW}, who proved a similar result but with $\log \max\{m,n\}$ replaced by the weaker expression $o(\max\{m,n\})$.
\begin{Th}\label{recgood}
Let
$$\displaystyle F(m)=\sum_{i=1}^{s}p_{i}(m)\alpha_{i}^{m}$$
$$\displaystyle G(n)=\sum_{j=1}^{t}q_{j}(n)\beta_{j}^{n}$$
define two distinct algebraic linear recurrence sequences, where $p_{i}$ and $q_{j}$ are polynomials. Let $k$ be a number field such that all coefficients of $p_{i}$ and $q_{j}$ and $\alpha_{i}$, $\beta_{j}$ are in $k$, for $i=1,\ldots,s$, $j=1,\ldots,t$. Let
$$S_{0}=\{v \in M_{k}:\max\{|\alpha_{1}|_{v},\ldots,|\alpha_{s}|_{v},|\beta_{1}|_{v},\ldots,|\beta_{t}|_{v}\} < 1\}.$$
Then there are finitely many choices of nonzero integers $(a_{i},b_{i}, c_{i}, d_{i}),\  a_{i}c_{i} \neq 0$ such that all solutions $(m,n) \in \mathbb{N}^{2}$ of the inequality
\begin{align}\label{finstat}
\displaystyle \sum_{v \in M_{k} \setminus S_{0}}-\log^{-}\max\{|F(m)|_{v},|G(n)|_{v}\} > \epsilon \max\{m,n\}
\end{align}
are of the form:
$$(m,n)=(a_{i}t+b_{i},c_{i}t+d_{i})+(\mu_{1},\mu_{2}),\ |\mu_{1}|,|\mu_{2}| \ll \log t,\ t \in \mathbb{N},\ i=1,\ldots,r.$$
\end{Th}

\begin{proof}
Now let $\{u_{1},\ldots,u_{r}\}$ be a set of generators which generates the roots of $F$ and $G$ and assume that the $u_{i}$'s are multiplicatively independent (as in the proof of Theorem \ref{recbad1}). It follows from the first part of the proof of Theorem \ref{recbad2} (using the points $P_{m,n}=(m,u_{1}^{m},...,u_{r}^{m},n,u_{1}^{n},...,u_{r}^{n})$) that all but finitely many pairs $(m,n)$ that fail the above inequality either satisfy finitely many linear relations $(m,n)=(a_{i}t+b_{i}, c_{i}t+d_{i})$ or satisfy an exponential-polynomial equation coming from Schmidt’s Subspace Theorem:
$$\displaystyle \sum_{w=1}^{W}P_{w}(m,n)u_{1}^{ms_{1,w}+nt_{1,w}}\cdots u_{r}^{ms_{r,w}+nt_{r,w}}=0,$$
where $P_{w}(m,n)$ are non-zero polynomials in $m$ and $n$. After ignoring finitely many arithmetic progressions, we can assume that $(m,n)$ is not a zero of any $P_{w}$ by Lemma \ref{sinvar}.\\

Dividing by the first term, we need to study the solutions $(m,n)$ to the equation
\begin{align}\label{finexc}
    \displaystyle  \sum_{w=2}^{W}Q_{w}(m,n)u_{1}^{ms'_{1,w}+nt'_{1,w}}\cdots u_{r}^{ms'_{r,w}+nt'_{r,w}}=1 
\end{align}
where $Q_{w}=-P_{w}(m,n)/P_{1}(m,n)$.

As in Theorem \ref{recbad2}, we can estimate the non-$S$ contribution to the height of each term in $\eqref{finexc}$ by
\begin{align}\label{qest}
    h(Q_{w}(m,n))\leq R_{w}\max\{\log m,\log n\}+O(1)
\end{align}
for some constant $R_{w}$. On the other hand, we have the estimate
\begin{align}\label{twest}
\displaystyle
&h(Q_{w}(m,n)u_{1}^{ms'_{1,w}+nt'_{1,w}}\cdots u_{r}^{ms'_{r,w}+nt'_{r,w}}) \nonumber \\
&\geq c_{w}\max_{i}\{|ms'_{i,w}+nt'_{i,w}|\}-R_w\log\max\{m,n\}+O(1)
\end{align}
for some constant $c_{w}$.

In order to apply Corollary \ref{unieqc}, we need each summand to be in $k_{S,\delta}$ for some $\delta < \dfrac{1}{W(W+1)}$. So it suffices to require, for every $w$,
\begin{align}\label{logreg}
    \displaystyle C_{w}\max\{\log m,\log n\} \leq \max_{i}\{|ms'_{i,w}+nt'_{i,w}|\}
\end{align}
where $C_{w}=\dfrac{4R_{w}W(W+1)}{c_{w}}$.
For those $(m,n)$ satisfying $\eqref{logreg}$, we can apply Corollary \ref{unieqc} to $\eqref{finexc}$. But since
$$h(Q_{w}(m,n)u_{1}^{ms'_{1,w}+nt'_{1,w}}\cdots u_{r}^{ms'_{r,w}+nt'_{r,w}}) \to \infty \text{ as } \max\{m,n\} \to \infty,$$
this implies that there are only finitely many solutions $(m,n)$ of
$$\displaystyle \sum_{v \in M_{k} \setminus S_{0}}-\log^{-}\max\{|F(m)|_{v},|G(n)|_{v}\} > \epsilon \max\{m,n\}$$
satisfying $\eqref{logreg}$.

For pairs $(m,n)$ not satisfying $\eqref{logreg}$, there exists some $w_0$ and $i_0$ such that 
$$ (s'_{i_{0},w_{0}},t'_{i_{0},w_{0}})\neq (0,0)$$ 
and
$$\displaystyle C_{w_{0}}\max\{\log m,\log n\} \geq |ms'_{i_{0},w_{0}}+nt'_{i_{0},w_{0}}|.$$
In fact, since as previously observed, $\min\{m,n\}\gg \max\{m,n\}$ for solutions $(m,n)$ to $\eqref{finstat}$, we may assume $s'_{i_{0},w_{0}}t'_{i_{0},w_{0}}\neq 0$.

Fix such a pair $(m,n)$ and corresponding $w_0$ and $i_0$. Let $a=s'_{i_0,w_0}, b=t'_{i_0,w_0}$, and $t=\max\left\{\left\lfloor\frac{m}{b}\right\rfloor,\left\lfloor-\frac{n}{a}\right\rfloor\right\}$. Replacing $(a,b)$ by $(-a,-b)$ if necessary, we may assume that $a<0$ and $b>0$. We set $\mu_{1}=m-bt$ and $\mu_{2}=n+at$, so that $(m,n)=(bt,-at)+(\mu_{1},\mu_{2})$. Then clearly $\min\{|\mu_{1}|,|\mu_{2}|\}\leq \max\{|a|,|b|\}$ and so
\begin{align*}
\max\{|\mu_{1}|,|\mu_{2}|\}\ll |a\mu_{1}+b\mu_{2}|=|am+bn|\ll \max\{\log m,\log n\}\ll \log t,
\end{align*}
as desired.
\end{proof}

\section{Fuchs-Heintze's Integral zeros of exponential-polynomials}
To better shape the exceptional points inside the logarithmic region, we need some results for exponential-polynomial equations. In the recent work of Fuchs-Heintze \cite{FH}, one can write integral solutions of a exponential-polynomial equation in the form of linear recurrences. Here we use a slightly different version of their theorem. Though the proof is mostly remains identical, we include them here for completeness, following Fuchs-Heintze \cite{FH}.

Before we proving the main result, some lemmas are needed. 

\begin{Lemma}\label{htieq}
Let $k$ be a number field and $f(x)=b_{0}x^{d}+\cdots+b_{d} \in k[x]$ a polynomial with $b_{0}\neq 0$. Assume that $\xi \in k$ is a zero of $f$. Then we have
$$h(\xi)\leq h(b_{0}:b_{1}:\ldots:b_{d})+\log d.$$
\end{Lemma}

For a number field $k$ and a nonzero $x \in k^*$, we define 
$$h_{\bar{S},0}(x)=\sum_{x \not \in S}\log^+|x|_v$$ 
and 
$$h_{\bar{S},\infty}(x)=\sum_{x \not \in S}\log^+|1/x|_v.$$ 
Note that $h_{\bar{S},0}(x)+h_{\bar{S},\infty}(x)$ is what we defined for $h_{\bar{S}}(x)$ in the first section. 

Now fix an absolute value $v$ in $M_k$. 
\begin{Lemma}[Corvaja-Zannier \cite{CZ2}]\label{CZ05b}
Let $ f(\textbf{x})=\sum_{\textbf{i}}a_{\textbf{i}}\textbf{x}^{\textbf{i}}$ be a power series with algebraic coefficients in $\mathbb{C}_{v}$ converging in a neighborhood of the origin in $\mathbb{C}_{v}^{k}$. Let $S$ be a finite set of absolute values of $K$ containing the archimedean ones. Let $\textbf{x}_{n}=(x_{n1},\ldots,x_{nl})$ ($n \in \mathbb{N}_{>0}$) be a sequence in $(k^*)^l$, tending to $0$ in $k_{v}^{n}$, and such that $f(\textbf{x}_{n})$ is defined and belongs to $k$. Suppose that
\begin{enumerate}
    \item For $i=1,\ldots, l$ we have $h_{\bar{S},0}(x_{ni})+h_{\bar{S},0}(x_{ni}^{-1})=o(h(x_{ni}))$ as $n\to \infty$.
    \item $\displaystyle \hat{h}(\textbf{x}_{n})=O(-\log \max_{i}|x_{ni}|_{v})$.
    \item $h_{\bar{S},0}(f(\textbf{x}_{n}))=o(h(\textbf{x}_{n}))$.
    \item $h(f(\textbf{x}_{n}))=O(h(\textbf{x}_{n}))$.
\end{enumerate}
Then there exists a finite number of cosets $\textbf{u}_{1}H_{1},\ldots, \textbf{u}_{t}H_{t} \subset \mathbb{G}_{m}^{k}$, where $H_i$ are connected algebraic subgroups of $\mathbb{G}_m^k$, such that $\{\textbf{x}_{n}\}_{n \in \mathbb{N}} \subset \bigcup_{i=1}^{t}\textbf{u}_{i}H_{i}$ and such that, for $i=1,\ldots, t$, the restriction of $f(\textbf{x})$ to $\textbf{u}_{i}H_{i}$ coincides with a polynomial in $k[\textbf{x}]$.
\end{Lemma}

\begin{Lemma}[Implicit Function Theorem]\label{IFT}
Suppose the power series
$$F(x_{1},\ldots,x_{r},y) =\sum_{|\alpha|\geq 0,k\geq 0}a_{\alpha,k}x_{1}^{\alpha_{1}}\cdots x_{r}^{\alpha_{r}}y^{k}$$
is absolutely convergent for $|x_{1}|+\cdots+|x_{r}|\leq R_{1}$, $|y-y_{0}|\leq R_{2}$ for some $y_{0} \in \bar{\mathbb{Q}}$ with $F(0,\ldots,0,y_{0})=0$. If 
$$\dfrac{\partial F}{\partial y}(0,\ldots,0,y_{0})\neq 0,$$
then there exists $r_{0} >0$ and a power series 
$$f(x_{1},\ldots,x_{r})=\sum_{|\alpha|>0}c_{\alpha}x_{1}^{\alpha_{1}}\cdots x_{r}^{\alpha_{r}}$$
which is absolutely convergent for $|x_{1}|+\cdots+|x_{r}| \leq r_{0}$ and 
$$F(x_{1},\ldots,x_{r},f(x_{1},\ldots, x_{r}))=0.$$
Moreover, if the coefficients of $F$ are algebraic, then the coefficients of $f$ are also algebraic.
\end{Lemma}

Denote by $\mathcal{H}(k)$ the Hadamard ring over $k$; i.e., the set of sequences in $k$ satisfying a linear recurrence relation.
\begin{Th}[Hadamard Quotient Theorem]
Let $k$ be a field of characteristic zero and let $b(n),c(n) \in \mathcal{H}(k)$. Let $a(n)$ be a sequence whose elements are in a subring $R$ of $k$ which is finitely generated over $\mathbb{Z}$, and suppose that $a_n=\frac{b(n)}{c(n)}$ whenever the quotient is defined. Then there exists an element $a(n) \in \mathcal{H}(k)$ such that $a(n)=a_n$ for every $n$ provided $c(n) \neq 0$. 
\end{Th}

Then we have a variant version of Theorem 1 in Fuchs-Heintze's situation.
\begin{Th}\label{FH1}
Let $k$ be a number field and $g \in k[x_{0},x_{1},\ldots,x_{r},z]$ be a polynomial which can be written in the form
$$g(x_{0},x_{1},\ldots,x_{r},z)=a_{0}(x_{0},x_{1},\ldots,x_{r})z^{d}+\cdots+a_{d}(x_{0},x_{1},\ldots,x_{r})$$
for polynomails $a_{0}(x_{0},x_{1},\ldots,x_{r}),\ldots,a_{d}(x_{0},x_{1},\ldots,x_{r})$. 

Furthermore, let $\tilde{g} \in k[x_{0},x_{1},\ldots,x_{r},z]$ be the polynomial given by the equation
$$\tilde{g}(x_{0},x_{1},\ldots,x_{r},\tilde{z})=a_{0}(x_{0},x_{1},\ldots,x_{r})^{d-1}g(x_{0},x_{1},\ldots,x_{r},z),$$
where $\tilde{z}=a_{0}(x_{0},x_{1},\ldots,x_{r})z$.
Assume that either $a_{0}(0,\ldots,0)\neq 0$ and $g(0,\ldots,0,z)$ has no multiple zero as a polynomial in $z$, or $a_{0}(0,\ldots,0)=0$ and $\tilde{g}(0,\ldots,0,\tilde{z})$ has no multiple zero as a polynomial in $\tilde{z}$. Moreover, let $\gamma_{1},\ldots, \gamma_{r} \in k^{*}$ such that $|\gamma_{i}|<1$ for all $1\leq i \leq r$, where $|\cdot|$ denotes the usual absolute value on $\mathbb{C}$, and such that no ratio $\gamma_{i}/\gamma_{j}$ for $i\neq j$ is a root of unity. Assume that $S$ is a finite set of places of $k$, containing all archimedean ones, and such that $\gamma_{1},\ldots,\gamma_{r}$ and all non-zero coefficients of $a_{i}(x_{0},\ldots,x_{r})$ for $i=0,\ldots,d$ are $S$-units.\\
Then there are finitely many polynomials $P_{1},\ldots, P_{t}$ in $r+1$ unknowns such that the following holds: For each solution $(n,z) \in \mathbb{N} \times \mathcal{O}_{S}$ of $g(n\gamma_{1}^{n},\gamma_{1}^{n},\ldots, \gamma_{r}^{n},z)=0$ with $z \neq 0$ and $n$ large enough, there exists an index $i$ such that $z'=P_{i}(n\gamma_{1}^{n},\gamma_{1}^{n},\ldots,\gamma_{r}^{n})$, where $z'= z$ in the case that $a_{0}(0,\ldots,0)\neq 0$ and $z'=a_{0}(n\gamma_{1}^{n},\gamma_{1}^{n},\ldots,\gamma_{r})^{n}z$ in the case that $a_{0}(0,\ldots,0)=0$, respectively.
\end{Th}

\begin{proof}
We write 
$$g(n\gamma_{1}^{n},\gamma_{1}^{n},\ldots,\gamma_{r}^{n},z)=a_{0}(n\gamma_{1}^{n},\gamma_{1}^{n},\ldots,\gamma_{r}^{n})z^{d}+\cdots+a_{d}(n\gamma_{1}^{n},\gamma_{1}^{n},\ldots,\gamma_{r}^{n})$$
and
$$\tilde{g}(n\gamma_{1}^{n},\gamma_{1}^{n},\ldots,\gamma_{r}^{n},\tilde{z})=\tilde{a_{0}}(n\gamma_{1}^{n},\gamma_{1}^{n},\ldots,\gamma_{r}^{n})\tilde{z}^{d}+\cdots+\tilde{a_{d}}(n\gamma_{1}^{n},\gamma_{1}^{n},\ldots,\gamma_{r}^{n}).$$

Since no ratio $\gamma_{i}/\gamma_{j}$ is a root of unity, by Theorem \ref{SML}, for $n \in \mathbb{N}$ large enough we have $$a_{i}(n\gamma_{1}^{n},\gamma_{1}^{n},\ldots,\gamma_{r}^{n})\neq 0$$ 
for all $1\leq i\leq d$. As the $\tilde{a_{j}}(n\gamma_{1}^{n},\gamma_{1}^{n},\ldots,\gamma_{r}^{n})$ arise by construction as products of the $a_{i}(n\gamma_{1}^{n},\gamma_{1}^{n},\ldots,\gamma_{r}^{n})$, they are non-zero as well for $n$ large and all $1\leq j\leq d$. Thus, we will assume that $n$ is large enough such that all $a_{i}$ and $\tilde{a_{j}}$ are non-zero.\\
We break into two cases: Let us assume that $a_{0}(0,\ldots,0) \neq 0$ and that $g(0,\ldots,0,z)$ has only simple zeros. In this case we only work with $g$. The other case, when $a_{0}(0,\ldots,0)=0$ and $\tilde{g}(0,\ldots,0,z)$ has only simple zeros, works in the same way considering $\tilde{g}$ instead of $g$ with the transformation $\tilde{z}=a_{0}(n\gamma_{1}^{n},\gamma_{1}^{n},\ldots,\gamma_{r}^{n})z$ and $\tilde{g}$ is monic in $\tilde{z}$. Hence, we write down only the first case.\\

Consider now an infinite sequence $((n,z_{n}))_{n\in W}$ of solutions of the equation
\begin{align}
    \label{epeq}
    g(n\gamma_{1}^{n},\gamma_{1}^{n},\ldots,\gamma_{r}^{n},z)=0 
\end{align}
in $(n,z) \in \mathbb{N}\times \mathcal{O}_{S}$ with $z\neq 0$, where $W$ is an infinite subset of $\mathbb{N}$. Since for fixed $n$ there are at most $d$ possible values for $z$, all solutions are contained in finitely many such sequences. Therefore we restrict to one of them. To simplify the notation, we write $a_{i}(n)$ instead of $a_{i}(n\gamma_{1}^{n},\gamma_{1}^{n},\ldots,\gamma_{r}^{n})$.

First, let us show $(z_{n})$ is bounded. Since $z_{n}$ is a solution of (\ref{epeq}) we have
$$a_{0}(n)z_{n}^{d}+\cdots+a_{d}(n)=0.$$
Using $z_{n}\neq 0$ this is equivalent to 
$$a_{0}(n)z_n=-a_{1}(n)-\cdots-a_{d}(n)z_{n}^{-(d-1)}.$$
For $|z_{n}|>1$ this yields
$$|a_{0}(n)z_{n}|\leq |a_{1}(n)|+\cdots+|a_{d}(n)|.$$
Hence, we have an upper bound
$$|z_{n}|\leq \max\left\{1,\dfrac{|a_{1}(n)|+\cdots+|a_{d}(n)|}{|a_{0}(n)|}\right\}.$$
As $a_{0}(0,\ldots,0)\neq 0$, the denominator $|a_{0}(n)|$ is bounded away from zero, together with $|\gamma_{i}|<1$ and so the exponential parts dominate each $a_{i}$, we obtain the boundedness of the sequence $(z_{n})$.

We estimate, for those solutions $(n,z_n)$
$$\begin{aligned}
|g(0,\ldots,0,z_{n})|&=|g(0,\ldots,0,z_{n})-g(n\gamma_{1}^{n},\gamma_{1}^{n},\ldots,\gamma_{r}^{n},z_{n})|\\
&\leq \sum_{i=0}^{d}\underbrace{|a_{i}(0,\ldots,0)-a_{i}(n\gamma_{1}^{n},\gamma_{1}^{n},\ldots,\gamma_{r}^{n})|}_{n \to \infty, \text{ this} \to 0}\cdot|z_{n}|^{d-i}.
\end{aligned}$$
By the boundedness of $z_{n}$, we obtain $g(0,\ldots,0,z_n) \to 0$ as $n \to \infty$. Thus, $z_{n}$ lie in the union of arbitrary small neighborhoods of the solutions of $g(0,\ldots,0,z)=0$ for $n$ large enough. Therefore we can split the sequence into finitely many subsequences and restrict only on one infinite sequence $(z_{n})$ which converges to a solution $z_{*}$ of $g(0,\ldots,0,z)=0$.\\

Let us derive an upper bound for the height of $z_{n}$. By Lemma \ref{htieq}, we have 
$$\begin{aligned}h(z_{n})& \leq h(1:a_{0}(n):\ldots:a_{d}(n))+\log d\\
&\leq h(a_{0}(n))+\cdots+h(a_{d}(n))+\log d.
\end{aligned}$$
It turns out we need bounds for $h(a_{i}(n))$. We now estimate
$$\begin{aligned}
h(a_{i}(n))&=h\left(\sum_{k_{0},\ldots,k_{r}} \lambda_{k_{0},\ldots,k_{r}}^{(i)}(n\gamma_{1}^{n})^{k_{0}}\gamma_{1}^{nk_{1}}\cdots \gamma_{r}^{nk_{r}}\right)\\
&\leq \sum_{k_{0},\ldots,k_{r}}h(\lambda_{k_{0},\ldots,k_{r}}^{(i)}(n\gamma_{1}^{n})^{k_{0}}\gamma_{1}^{nk_{1}}\cdots \gamma_{r}^{nk_{r}})+C_{i,0}\\
&\leq \sum_{k_{0},\ldots,k_{r}}(h(\lambda_{k_{0},\ldots,k_{r}}^{(i)})+k_{0}h(n\gamma_{1}^{n})+k_{1}h(\gamma_{1}^{n})+\cdots+k_{r}h(\gamma_{r}^{n}))+C_{i,0}\\
&=C_{i,0}'+C_{i,0}h(n\gamma_{1}^{n})+C_{i,1}h(\gamma_{1}^{n})+\cdots+C_{i,r}h(\gamma_{r}^{n})\\
&\leq C_{i,0}'+(C_{i,0}+\cdots+C_{i,r})\max_{j=1,\ldots,r}\{h(n\gamma_{1}^{n}),h(\gamma_{j}^{n})\}\\
&\leq C_{i,0}'+C_{i}'h(1:n\gamma_{1}^{n}:\gamma_{1}^{n}:\ldots:\gamma_{r}^{n}),
\end{aligned}$$
where the constants depend on $a_{i}$. Combining these two inequalities, we have the upper bound 
$$h(z_{n})\leq C_{1}+C_{2}h(1:n\gamma_{1}^{n}:\gamma_{1}^{n}:\ldots:\gamma_{r}^{n}).$$

Now we are ready to show the power series form and even the polynomial form of $z_{n}$. Note that $g(0,\ldots,0,z)$ has no multiple zero as a polynomial in $z$ is equivalent to 
$$\dfrac{\partial g}{\partial z}(0,\ldots,0,z_{*})\neq 0$$
for all $z_{*}$ satisfying $g(0,\ldots,0,z_{*})=0$. Then we apply the Lemma \ref{IFT}, to conclude that there is a power series $f(x_{0},\ldots,x_{r})$, for $n$ sufficiently large,
$$z_{n}=f(n\gamma_{1}^{n},\gamma_{1}^{n},\ldots,\gamma_{r}^{n}).$$
Next we would like to apply Lemma \ref{CZ05b}. We check the conditions. We have a power series $f(x_{0},\ldots,x_{r})$ with algebraic coefficients converging in a neighborhood of the origin with respect to the standard absolute value in $\mathbb{C}$. Define the vector $\textbf{x}_{n}$ to be $(w_{n}\gamma_{1}^{w_{n}},\gamma_{1}^{w_{n}},\ldots,\gamma_{r}^{w_{n}})$, where $w_{n}$ is the $n$-th index in the set $W$. Then $\textbf{x}_{n}$ tends to zero, and $f(\textbf{x}_{n})=z_{w_{n}}$ is defined and belongs to $k$.\\

It remains to check the four height conditions. For $i\neq 0$, $x_{ni}=\gamma_{i}^{w_{n}}$ is an $S$-unit, so $h_{\bar{S},0}(x_{ni})+h_{\bar{S},0}(x_{ni}^{-1})=0$. For $i=0$, $h_{\bar{S},0}(x_{n0})+h_{\bar{S},0}(x_{n0}^{-1}) = h_{\bar{S},0}(w_{n})+h_{\bar{S},0}(w_{n}^{-1}) \leq h(w_{n})+h(w_{n}^{-1})$, and as $w_{n} \to \infty$, we have $h(w_{n})+h(w_{n}^{-1})=o(h(w_{n}\gamma_{1}^{w_{n}}))$. So the first condition is satisfied.

Consider, for sufficiently large $w_{n}$,
$$\hat{h}(w_{n}\gamma_{1}^{w_{n}},\gamma_{1}^{w_{n}},\ldots,\gamma_{r}^{w_{n}})=h(w_{n}\gamma_{1}^{w_{n}})+\sum_{i=1}^{r}h(\gamma_{i}^{w_{n}})\leq 3w_{n}h(\gamma_{1})+w_{n}\sum_{i=2}^{r}h(\gamma_{i})=w_{n}C_{3}.$$
On the other hand, we also have
$$-\log(\max_{i}|x_{ni}|) = -\log (\max\{|w_n\gamma_1^{w_n}|,|\gamma_{i}^{w_{n}}|\}) \geq w_{n}(-\log (\max_{i} |\gamma_{i}|))-\log w_n \geq w_{n}C_{4}.$$
Hence the second condition is satisfied.

For the third condition,
$$h_{\bar{S},0}(f(\textbf{x}_{n}))=h_{\bar{S},0}(z_{w_{n}}),$$
by our assumption, $z_{w_{n}}$ is an $S$-integer. Hence, $h_{\bar{S},0}(z_{w_{n}})=0$.

The last condition is automatic from the inequality we already have
$$h(z_{w_{n}})\leq C_{1}+C_{2}h(1:w_{n}\gamma_{1}^{w_{n}}:\gamma_{1}^{w_{n}}:\ldots:\gamma_{r}^{w_{n}}).$$

Now we apply Lemma \ref{CZ05b} and it gives us finitely many cosets $\textbf{u}_{1}H_{1},\ldots,\textbf{u}_{t}H_{t}$ in $\mathbb{G}_{m}^{r+1}$ such that $\{w_{n}\gamma_{1}^{w_{n}},\gamma_{1}^{w_{n}},\ldots,\gamma_{r}^{w_{n}}\}_{n \in \mathbb{N}} \subset \bigcup_{i=1}^{t}\textbf{u}_{i}H_{i}$ and such that for $i=1,\ldots,t$, the restriction of $f$ to $\textbf{u}_{i}H_{i}$ coincides with a polynomial $P_{i}$ in $k[x_{0},\ldots,x_{r}]$. Hence, for all $n \in W$, there exists an index $i$ such that $(n\gamma_{1}^{n},\gamma_{1}^{n},\ldots,\gamma_{r}^{n}) \in \textbf{u}_{i}H_{i}$ and $z_{n}=P_{i}(n\gamma_{1}^{n},\gamma_{1}^{n},\ldots,\gamma_{r}^{n})$.\\

The second case is similar to the first case.
\end{proof}

We are now ready to formulate the next theorem, which is a variant of Theorem 2 of Fuchs-Heintze:
\begin{Th}\label{FH1}
Let $K,g,\tilde{g},\gamma_{1},\ldots,\gamma_{r}$ as before. Then there are finitely many linear recurrences $R_{1}(n),\ldots, R_{s}(n)$ with algebraic roots and algebraic coefficients, arithmetic progressions $\mathcal{P}_{1},\ldots,\mathcal{P}_{s}$, as well as finite sets $M$ and $N$ such that the set $L$ of solutions $(n,z) \in \mathbb{N}\times \mathcal{O}_{S}$ of the equation $g(n\gamma_{1}^{n},\gamma_{1}^{n},\ldots,\gamma_{r}^{n},z)=0$ can be described by
$$L=\bigcup_{j=1}^{s}\{(n,R_{j}(n)):n \in \mathcal{P}_{j},R_{j}(n) \in \mathcal{O}_{S}\}\cup \{(n,z):n \in N,z \in \mathcal{O}_{S}\}\cup M.$$
\end{Th}

\begin{proof}
First note that for a fixed value of $n$, the equation
\begin{align}\label{epeq2}
    g(n\gamma_{1}^{n},\gamma_{1}^{n},\ldots,\gamma_{r}^{n},z)=0
\end{align}
has either finitely many solutions $z$ if not all $a_{i}(n)$ are zero, or holds for all values of $z$ if all $a_{i}(n)$ are zero. So for finitely many $n$ of the solutions $(n,z) \in \mathbb{N}\times \mathcal{O}_{S}$, these two cases fit into the latter two patterns of the solutions. From now on, we assume $n$ is sufficiently large.\\

For $z = 0$, the equation reduces to $a_{d}(n)=0$, which has only finitely many solutions in $n$ since it is a non-degenerate linear recurrence, or is identically zero which fits the pattern $\{(n,0):n\in \mathbb{N}\}$. Therefore, we assume $z\neq 0$.\\

It remains to classify the rest of the solutions. Applying Theorem \ref{FH1}, there are finitely many polynomials $P_{1},\ldots,P_{t}$ such that for all remaining solutions $(n,z)$ there is an index $i \in \{1,\ldots,t\}$ with the property that either
\begin{align}\label{epc1}
    z=P_{i}(n\gamma_{1}^{n},\gamma_{1}^{n},\ldots,\gamma_{r}^{n}),
\end{align}

or
\begin{align}\label{epc2}
    z=\dfrac{P_{i}(n\gamma_{1}^{n},\gamma_{1}^{n},\ldots,\gamma_{r}^{n})}{a_{0}(n\gamma_{1}^{n},\gamma_{1}^{n},\ldots,\gamma_{1}{n})}.
\end{align}

Now we have four cases: For finitely many solutions of pattern (\ref{epc1}) and (\ref{epc2}), they fit into $M$ are already two of the fours cases.\\

The third case is when for an fixed index $i$ such that there are infinitely many solutions $(n,z)$ of pattern (\ref{epc1}), $z$ is in the form of a given linear recurrence sequence. Note that (\ref{epeq2}) is saying a linear recurrence is zero. By Theorem \ref{SML}, we have the property of $n$ on some arithmetic progressions with $z$ a linear recurrence as well.\\

The last case (\ref{epc2}), we only need to do apply the Hadamard quotient theorem. Since $z$ is an $S$-integer and the ring of $S$-integers is finitely generated over $\mathbb{Z}$, so the conditions of the Hadamard quotient theorem are satisfied. This give $z$ a form of a linear recurrence again. The rest of the argument is similar to the above paragraph.
\end{proof}

\section{The exceptional case}
We give the exceptional case a further description, taking the advantage of the Fuchs-Heintze's theorems. First, we show that the exceptional cases occurs in the Theorem \ref{recgood} only if $m,n$ are in the form of linear recurrences. Moreover, in the exceptional cases, by appropriate coordinate change, $F$ and $G$ have a nontrivial common factor. 
\begin{Th}\label{finimp}
Let
$$\displaystyle F(m)=\sum_{i=1}^{s}p_{i}(m)\alpha_{i}^{m}$$
$$\displaystyle G(n)=\sum_{j=1}^{t}q_{j}(n)\beta_{j}^{n}$$
define two distinct algebraic linear recurrence sequences, where $p_{i}$ and $q_{j}$ are polynomials. Let $k$ be a number field such that all coefficients of $p_{i}$ and $q_{j}$ and $\alpha_{i}$, $\beta_{j}$ are in $k$, for $i=1,\ldots,s$, $j=1,\ldots,t$. Let
$$S_{0}=\{v \in M_{k}:\max\{|\alpha_{1}|_{v},\ldots,|\alpha_{s}|_{v},|\beta_{1}|_{v},\ldots,|\beta_{t}|_{v}\} < 1\}.$$
Then there are finitely many choices of nonzero integers $(a_{i},b_{i}, c_{i}, d_{i}),\  a_{i}c_{i} \neq 0$ such that all but finitely many solutions $(m,n) \in \mathbb{N}^{2}$ of the inequality
$$\displaystyle \sum_{v \in M_{k} \setminus S_{0}}-\log^{-}\max\{|F(m)|_{v},|G(n)|_{v}\} > \epsilon \max\{m,n\},$$
either satisfy finitely many linear relations:
$$(m,n)=(a_{i}t+b_{i},c_{i}t+d_{i}), i=1,\ldots,r,$$
or there exist a pair of constants $(a,b)$ with $T:=|am+bn| \ll O(\max\{\log m,\log n\})$ and linear recurrences $f$ and $g$ indexed by $T$ such that $m=f(T)$ and $n=g(T)$. \\
Moreover, assume $\{u_i\}_{i=1,\ldots,y_r}$ is the set of the combined roots of $F,G,m,n$ such that $F$ and $G$ can be written as polynomials in variables $T,x_1,\ldots,x_r,y_1,\ldots, y_r$, where $x_i=u_i^T$ and $y_i=u_i^m$. Then $F,G$ admit a non-trivial common divisor in $k[T,x_1,\ldots,x_r,y_1,\ldots,y_r]$.
\end{Th}

\begin{proof}
Let $u_{1},\ldots,u_{r}$ be the generators of the torsion-free group generated by the combined roots of $F$ and $G$. Without loss of generality, choose them such that $|u_{i}| \leq 1$. From the proof of Theorem \ref{recgood}, we get that the exceptional set is covered by
\begin{align}\label{impexc}
\sum_{w}Q_{w}(m,n)u_{1}^{ms_{1,w}+nt_{1,w}}\cdots u_{r}^{ms_{r,w}+nt_{r,w}}=1,
\end{align}
where $s_{i,w}, t_{i,w} \in \mathbb{Z}$. Let us denote the $w$-th term in $\eqref{impexc}$ by $T_{w}$. Again by Theorem \ref{recgood}, for such $(m,n)$, there exists a $w_{0}$ with the property:
$$\mathcal{P}:\ \max_{i}\{|ms_{i,w_{0}}+nt_{i,w_{0}}|\}\leq O(\max\{\log m,\log n\}).$$
Then the case splits into
\begin{enumerate}[(a)]
    \item All $w\in W$ satisfy $\mathcal{P}$.
    \item At least one of $w \in W$ does not satisfy $\mathcal{P}$.
\end{enumerate}

In case (b), partition indices $w$ according to $\mathcal{P}$: Let 
$$\bar{W_{1}}=\{w|T_{w} \text{ satisfies }\mathcal{P} \},\ \bar{W_{2}}=\{w|T_{w} \text{ does not satisfy }\mathcal{P}\}.$$
Then, we can write $\eqref{impexc}$ as 
$$\sum_{w \in \bar{W_{2}}}T_{w}=1-\sum_{w \in \bar{W_{1}}}T_{w}.$$
Note that the left hand side are all almost $(S,\delta)$-units and the right hand side has small height. Thus
$$\dfrac{\sum_{w\in \bar{W_{2}}}T_{w}}{1-\sum_{w \in \bar{W_{1}}}T_{w}}=1$$
is an almost $(S,\delta)$-unit equation. Indeed, if we denote by $S_w$ the $w$-th summand $\dfrac{T_{w}}{1-\sum_{w' \in \bar{W_{1}}}T_{w'}}$, then the above equation is
\begin{align}\label{simpunit}
    \sum_{w\in \bar{W_{2}}}S_w=1.
\end{align}
We shall show that $S_w$ is an $(S,\delta)$-unit for each $\delta >0$ and each $w \in \bar{W}_2$ provided $\min\{m,n\}$ or $\max\{m,n\}$ sufficiently large. Note that 
$$h(S_w) \geq h(T_w), \ w \in \bar{W}_2.$$
By (\ref{twest}) and the definition of $\bar{W}_2$, there exists constants $c_w$ and $c'_w$ such that for $w \in \bar{W}_2$
$$h(T_w) \geq c_w\max_i\{|ms_{i,w}+nt_{i,w}|\} \geq c'_w \max\{m,n\}.$$
On the other hand, by (\ref{qest}), there exists constants $R_w$ and $R'_w$ such that
$$\begin{aligned}
h_{\bar{S}}(S_w) &\leq h_{\bar{S}}(T_w)+h_{\bar{S}}(1-\sum_{w \in \bar{W_{1}}}T_{w})\\
&\leq 2R_w\max\{\log m,\log n\} + 2h(1-\sum_{w \in \bar{W_{1}}}T_{w})+O(1)\\
&\leq 2R_w\max\{\log m,\log n\} + \sum_{w \in \bar{W_{1}}}h(T_w) + O(1)\\
&\leq (2R_w+R'_w)\max\{\log m,\log n\} +O(1).
\end{aligned}$$
The last inequality holds from the definition of $\bar{W}_1$. Hence, $h_{\bar{S}}(S_w) \leq \delta h(S_w)$ holds for any small positive $\delta$ with sufficiently large $m,n$ and so $S_w$ is an $(S,\delta)$-unit. Now we apply Corollary \ref{unieqc} to (\ref{simpunit}), to obtain that there are only finitely many values for each $S_w$. But when $\min\{m,n\}$ or $\max\{m,n\} \to \infty$, we have $h(S_w) \to \infty$. That results in only finitely many choices of $m,n$.\\

Now it remains to consider (a). For sufficiently large $m,n$, let $T=ms_{w_{0},i_{0}}+nt_{w_{0},i_{0}}$ for some $w_{0},i_{0} \in \mathbb{Z}$ satisfying $T \ll \log \max\{m,n\}$. Note that for any $w_j,i_j$, we have that $ms_{w_{j},i_{j}}+nt_{w_{j},i_{j}}$ must be a constant multiple of $T$, otherwise it will contradicts to its logarithmically small size. Hence, we can write $\eqref{impexc}$ as
$$\sum_{w}Q_{w}(m)u_{1}^{c_{1,w}T}\cdots u_{r}^{c_{r,w}T}-1=0,$$
where $c_{i,w}$ are constants. Combining terms and multiplying across by the denominators denominators, we could rewrite the equation as
$$L_{d}(T)m^{d}+\cdots+L_{0}(T)=0,$$
where $L_i(T)$ are linear recurrences in $T$.
For the simplicity of the notation, we could write the equation as
$$g(T,v_{1}^{T},\ldots,v_{r'}^{T},m)=0,$$
where $g$ is a polynomial in $k[x_{0},\ldots,x_{r'},z]$ and $v_{j}$ are the combined roots of $L_{i}(T)$.\\

Here we assume that all $|v_{j}| \leq 1$. We split the analysis once again:
\begin{enumerate}[(i)]
    \item All $|v_{j}| < 1$.
    \item At least one of $v_{j}$ has $|v_{j}|=1$.
\end{enumerate}
For case (ii), we can write $v_{j}=v_{j}'^{-1}v_{j}''$, where both $|v_{j}'|$ and $|v_{j}''|$ are less than $1$. Then we write our $g$ as
$$g=g(T,v_{1}^{T},\ldots,\hat{v_{j}}^{T},v_{j}'^{T},v_{j}''^{T},v_{j+1}^{T},\ldots,v_{r'}^{T},m)=0,$$
where $\hat{v}_j$ means we omit $v_j$. This is a Laurent polynomial, but one can always turn it into a polynomial by multiplying it with an appropriate monomial. Thus from now on we reduce all cases to case (i).

Now we define $g_{0}$ as follows. Take the equation $g(x_{0},\ldots,x_{r'},x_{r'+1})=0$, suppose there exists a pair $(T,m)\in \mathbb{N}^{2} \subset \mathbb{N}\times \mathcal{O}_{S}$ such that $(T,v_{1}^{T},\ldots,v_{r'}^{T},m)$ is a zero of $g$. Let the rational function $g'$ be
$$g'(x_{0},\ldots,x_{r'},x_{r'+1})=g(x_{0}/x_{1},x_{1},\ldots,x_{r'},x_{r'+1}),$$
which is a Laurent polynomial, then we polynomialize it
$$g_{0}(x_{0},\ldots,x_{r'+1})=g'(x_{0},\ldots,x_{r'+1})x_{1}^{\deg_{x_{0}}g}.$$
Hence, we obtain a polynomial $g_{0} \in k[x_{0},\ldots,x_{r'+1}]$ with a zero $(Tv_{1}^{T},v_{1}^{T},\ldots,v_{r'}^{T},m)$. 

Apply the Theorem \ref{FH1}, we have, excluding a finite set and finitely many lines, $m$ has the form of a general linear recurrence in $T$. Similarly by symmetry, we obtain that $n$ also has the form of a general linear recurrence in $T$.\\

Now we make a new choice of $u_{i}$ to be the generators of the torsion-free group generated by roots of $F,G$ and $m,n$, such that each root is a power of those generators with positive exponents. We write $m(T)$ and $n(T)$ for the linear recurrence indexed by $T$. Let $x_{i}=u_{i}^{T}$ and $y_{i}=u_{i}^{m}$. Therefore, we can write $F$ and $G$ as polynomials in coordinates $[T,x_{1},\ldots,x_{r},y_{1},\ldots,y_{s}]$:
$$F(m(T))=\sum_{i}c_{i}(m(T))\prod_{l} y_{l}^{p_{l}},$$
$$G(n(T))=\sum_{j}d_{j}(n(T))\prod_{l'}x_{l'}^{p_{l'}}y_{l'}^{q_{l'}}.$$
Since $T \ll \log \max\{m,n\}$ and $m,n$ are positive integers, the $s_{w_0,i_0}$ and $t_{w_0,i_0}$ defining $T$ should be in opposite signs. Without loss of generality, we put $s_{w_0,i_0} < 0$ and $t_{w_0,i_0} > 0$. Furthermore, after possibly replacing $T$ by $T/t_{w_0,i_0}$, we assume $t_{w_0,i_0}=1$. Note that under such conditions $F$ and $G$ are indeed polynomials.\\

We claim that $F$ and $G$ in above polynomial forms must have a non-trivial gcd. Suppose they don't, then by Theorem \ref{ins}, we have the exceptional set is covered by
$$\sum_{i,j,t}c_{i,j,t}T^{t}\prod_{i,j}x_{r_{i}}^{p_{r_{i}}}y_{r_{j}}^{q_{r_{j}}}=0. $$
By moving and dividing by some term, we have 
$$\sum_{i,j,t}c_{i,j,t}T^{t}\prod_{i,j}x_{r_{i}}^{p_{r_{i}}}y_{r_{j}}^{q_{r_{j}}}=1.$$
Together with the definition of $x_{i}$ and $y_{i}$, after possibly extend $S$, we assume the coefficients of the equation and $x_{i},y_{j}$ are $S$-units. This implies, for each summand, its non-$S$ height only comes from the part $T^{t}$. 

Since $T^{t}\ll \log x_{r_{i}}^{p_{r_{i}}} \ll \log \log y_{r_{j}}^{q_{r_{j}}}$, when $T$ is sufficiently large, a summand fails to be an almost $(S,\delta)$-unit if and only if $p_{r_{i}}$ and $q_{r_{j}}$ are simultaneously zero. We can combine all terms with $p_{r_{i}}=q_{r_{j}}=0$, we have rational functions in $T$ instead of a power of $T$ in each term on the left. This would make every summand an almost $(S,\delta)$-unit for large enough $T$. Hence we assume from now, $p_{r_{i}}q_{r_{j}}\neq 0$. Apply the almost $(S,\delta)$-unit equation in Corollary\ref{unieqc}, we obtain that, for each summand, there are only finitely many solutions. 

Without loss of generality, let us assume $m\geq n$, then for any summand, there exist constants $a_i,b_i$ such that
\begin{align}\label{sizeT}\begin{aligned}c_{i,j,t}T^{t}\prod_{i,j}x_{r_{i}}^{p_{r_{i}}}y_{r_{j}}^{q_{r_{j}}}&=c_{i,j,t}T^{t}\exp(a_{i}m+b_{i}n)\\
&=c_{i,j,t}T^{t}\exp \left(\dfrac{b_{i}}{t_{w_{0},i_{0}}}a_{i}m+b_{i}n+(a_{i}-\dfrac{b_{i}}{t_{w_{0},i_{0}}})m\right)\\
&=c_{i,j,t}T^{t}\exp\left(\dfrac{b_{i}}{t_{w_{0},i_{0}}}T+(a_{i}-\dfrac{b_{i}}{t_{w_{0},i_{0}}})m\right).
\end{aligned}\end{align}
Now since $T \ll \max\{\log m,\log n\}$, then we know for sufficiently large $T$, $$\dfrac{b_{i}}{t_{w_{0},i_{0}}}T+(a_{i}-\dfrac{b_{i}}{t_{w_{0},i_{0}}})m$$ 
never vanish. Hence, \eqref{sizeT} tends to $\infty$ and contradicts the finiteness. Hence, $T$ can only have finitely values and this implies that $(m,n)$ satisfies finitely many linear relations.

\end{proof}

\textbf{Acknowledgments.} The author would like to thank Aaron Levin for many helpful comments on several proofs of the main theorem as well as helping polish this paper. The author would like to thank Nathan Grieve and Joseph Silverman for advises on writing and the reference list. The author also would like to thank the anonymous referee for suggestions on possible improvements.

\bibliographystyle{amsalpha}
\bibliography{GREATEST_COMMON_DIVISORS_FOR_POLYNOMIALS_IN_ALMOST_UNITS_AND_APPLICATIONS_TO_LINEAR_RECURRENCE_SEQUENCES}

\providecommand{\bysame}{\leavevmode\hbox to3em{\hrulefill}\thinspace}
\providecommand{\MR}{\relax\ifhmode\unskip\space\fi MR }
\providecommand{\MRhref}[2]{%
  \href{http://www.ams.org/mathscinet-getitem?mr=#1}{#2}
}
\providecommand{\href}[2]{#2}
\begin{thebibliography}{EvdPSW03}

\bibitem[AV09]{AV}
Daniel Allcock and Jeffrey~D. Vaaler, \emph{A {B}anach space determined by the
  {W}eil height}, Acta Arith. \textbf{136} (2009), no.~3, 279--298.
  \MR{2475695}

\bibitem[BCZ03]{BCZ}
Yann Bugeaud, Pietro Corvaja, and Umberto Zannier, \emph{An upper bound for the
  {G}.{C}.{D}. of {$a^n-1$} and {$b^n-1$}}, Math. Z. \textbf{243} (2003),
  no.~1, 79--84. \MR{1953049}

\bibitem[BG06]{BG}
Enrico Bombieri and Walter Gubler, \emph{Heights in {D}iophantine geometry},
  New Mathematical Monographs, vol.~4, Cambridge University Press, Cambridge,
  2006. \MR{2216774}

\bibitem[CLZ]{CLZ}
Pietro Corvaja, Aaron Levin, and Umberto Zannier, \emph{Greatest common
  divisors over characteristic zero function fields}, preprint.

\bibitem[CZ03]{CZ}
P.~Corvaja and U.~Zannier, \emph{On the greatest prime factor of
  {$(ab+1)(ac+1)$}}, Proc. Amer. Math. Soc. \textbf{131} (2003), no.~6,
  1705--1709. \MR{1955256}

\bibitem[CZ05a]{CZ1}
Pietro Corvaja and Umberto Zannier, \emph{A lower bound for the height of a
  rational function at {$S$}-unit points}, Monatsh. Math. \textbf{144} (2005),
  no.~3, 203--224. \MR{2130274}

\bibitem[CZ05b]{CZ2}
\bysame, \emph{{$S$}-unit points on analytic hypersurfaces}, Ann. Sci.
  \'{E}cole Norm. Sup. (4) \textbf{38} (2005), no.~1, 76--92. \MR{2136482}

\bibitem[EvdPSW03]{EPSW}
Graham Everest, Alf van~der Poorten, Igor Shparlinski, and Thomas Ward,
  \emph{Recurrence sequences}, Mathematical Surveys and Monographs, vol. 104,
  American Mathematical Society, Providence, RI, 2003. \MR{1990179}

\bibitem[Eve84]{Ev2}
Jan-Hendrik Evertse, \emph{On sums of {$S$}-units and linear recurrences},
  Compositio Math. \textbf{53} (1984), no.~2, 225--244. \MR{766298}

\bibitem[Eve02]{Ev}
\bysame, \emph{Points on subvarieties of tori}, A panorama of number theory or
  the view from {B}aker's garden ({Z}\"{u}rich, 1999), Cambridge Univ. Press,
  Cambridge, 2002, pp.~214--230. \MR{1975454}

\bibitem[FH21]{FH}
Clemens Fuchs and Sebastian Heintze, \emph{Integral zeros of a polynomial with
  linear recurrences as coefficients}, Indag. Math. (N.S.) \textbf{32} (2021),
  no.~3, 691--703. \MR{4246133}

\bibitem[Fuc03]{F}
Clemens Fuchs, \emph{An upper bound for the {G}.{C}.{D}. of two linear
  recurring sequences}, Math. Slovaca \textbf{53} (2003), no.~1, 21--42.
  \MR{1964201}

\bibitem[Gri20]{Gr}
Nathan Grieve, \emph{Generalized {GCD} for toric {F}ano varieties}, Acta Arith.
  \textbf{195} (2020), no.~4, 415--428. \MR{4121877}

\bibitem[GW20]{GW}
Nathan Grieve and Julie Tzu-Yueh Wang, \emph{Greatest common divisors with
  moving targets and consequences for linear recurrence sequences}, Trans.
  Amer. Math. Soc. \textbf{373} (2020), no.~11, 8095--8126. \MR{4169683}

\bibitem[HL03]{HL}
Santos Hern\'{a}ndez and Florian Luca, \emph{On the largest prime factor of
  {$(ab+1)(ac+1)(bc+1)$}}, Bol. Soc. Mat. Mexicana (3) \textbf{9} (2003),
  no.~2, 235--244. \MR{2029272}

\bibitem[Lev19]{Le}
Aaron Levin, \emph{Greatest common divisors and {V}ojta's conjecture for
  blowups of algebraic tori}, Invent. Math. \textbf{215} (2019), no.~2,
  493--533. \MR{3910069}

\bibitem[Luc05]{Lu}
Florian Luca, \emph{On the greatest common divisor of {$u-1$} and {$v-1$} with
  {$u$} and {$v$} near {$\mathscr S$}-units}, Monatsh. Math. \textbf{146}
  (2005), no.~3, 239--256. \MR{2184226}

\bibitem[Sil87]{SilSInt}
Joseph~H. Silverman, \emph{A quantitative version of {S}iegel's theorem:
  integral points on elliptic curves and {C}atalan curves}, J. Reine Angew.
  Math. \textbf{378} (1987), 60--100. \MR{895285}

\bibitem[Sil05]{Si}
\bysame, \emph{Generalized greatest common divisors, divisibility sequences,
  and {V}ojta's conjecture for blowups}, Monatsh. Math. \textbf{145} (2005),
  no.~4, 333--350. \MR{2162351}

\end{thebibliography}

\Addresses
\end{document}